\documentclass[12pt]{article}
\usepackage[a4paper,mag=1000,includefoot,left=2cm,right=2cm,top=2cm,bottom=2cm,headsep=1cm,footskip=1cm]{geometry}
\usepackage[cp1251]{inputenc}
\usepackage[russian,english]{babel}
\usepackage[T2A]{fontenc}
\usepackage{amsfonts,amsmath,amssymb,amsthm}
\usepackage{bm}
\usepackage{cite}
\usepackage{pdflscape}
\usepackage{pgfplotstable,booktabs}
\usepackage{longtable,supertabular,multirow,array}
\def\ve{\varepsilon}
\def\leq{\leqslant}
\def\geq{\geqslant}
\def\*#1{\mathbf{#1}}
\def\onetwelve{{\textstyle\frac{1}{12}}}
\def\half{{\textstyle\frac12}}

\DeclareMathOperator{\sgn}{sgn}

\newtheorem{theorem}{\indent Theorem}
\newtheorem{lemma}{\indent Lemma}
\newtheorem{remark}{\indent Remark}
\newtheorem{corollary}{\indent Corollary}
\begin{document}
\begin{center}
{\Large\textbf{On compact 4th order finite-difference schemes for the wave equation}}
\end{center}
\begin{center}
\medskip\par\noindent{\textbf{Alexander Zlotnik}${}^{a}$
\footnote{\small Corresponding author.\\
 E-mail addresses: \text{azlotnik@hse.ru} (A. Zlotnik), \text{kireevaoi@rgsu.net} (O. Kireeva)},
\textbf{Olga Kireeva}${}^{b}$}
\medskip\par\noindent$^{a}$
{\small National Research University Higher School of Economics,\\ 109028 Pokrovskii bd. 11, Moscow, Russia
\par\noindent$^b$ {\small Russian State Social University, W. Pieck 4, 129226 Moscow, Russia}}
\end{center}
\begin{abstract}
\noindent We consider compact finite-difference schemes of the 4th approximation order for an initial-boundary value problem (IBVP) for the $n$-dimensional non-homogeneous wave equation, $n\geq 1$.
Their construction is accomplished by both the classical Numerov approach and alternative technique based on averaging of the equation, together with further necessary improvements of the arising scheme for $n\geq 2$.
The alternative technique is applicable to other types of PDEs including parabolic and time-dependent Schr\"{o}dinger ones.
The schemes are implicit and three-point in each spatial direction and time and include a scheme with a splitting operator for $n\geq 2$.
For $n=1$ and the mesh on characteristics, the 4th order scheme becomes explicit and close to an exact four-point scheme.
We present a conditional stability theorem covering the cases of stability in strong and weak energy norms with respect to both initial functions and free term in the equation.
Its corollary ensures the 4th order error bound in the case of smooth solutions to the IBVP.
The main schemes are generalized for non-uniform rectangular meshes.
We also give results of numerical experiments showing the sensitive dependence of the error orders in three norms
on the weak smoothness order of the initial functions and free term and essential advantages over the 2nd approximation order schemes in the non-smooth case as well.
 \end{abstract}

\section{Introduction}
\label{s:1}

Compact higher-order finite-difference schemes for PDEs is a popular subject and a vast literature is devoted to them.
The case of such type schemes for the wave equation have recently attracted a lot of interest, in particular, see  \cite{BTT18,HLZ19,LLL19,STT19}, where much more related references can be found.

We consider compact finite-difference schemes of the 4th approximation order for an initial-boundary value problem (IBVP) for the $n$-dimensional wave equation with constant coefficients, $n\geq 1$.
Their construction on uniform meshes is accomplished by both the classical Numerov approach and alternative technique based on averaging of the equation related to the polylinear finite element method (FEM), together with further necessary improvements of the arising scheme for $n\geq 2$.
This alternative technique is applicable to other types of PDEs including parabolic and time-dependent Schr\"{o}dinger equations (TDSE).
The constructed schemes are implicit and three-point in each spatial direction and time.
For $n\geq 2$, there is a scheme with a splitting operator among them.
Notice that we use implicit approximations for the second initial condition in the spirit of the approximations for the equation.
Curiously, for $n=1$ and the mesh on characteristics of the equation, the 4th order scheme becomes explicit and very close to an exact scheme on a four-point stencil.
\par We present a conditional stability theorem covering the cases of stability in strong (standard) and weak energy norms with respect to both initial functions and free term in the equation.
Its corollary rigorously ensures the 4th order error bound in the case of smooth solutions to the IBVP.
Note that stability is unconditional for similar compact schemes on uniform meshes for other type PDEs, for example, see
\cite{DZR15,S77}.
Our approach is applied in a unified manner for any $n\geq 1$ (not separately for $n=1$, 2 or 3 as in many papers),
the uniform rectangular (not only square) mesh is taken,
the stability results are of standard kind in the theory of finite-difference schemes and proved by the energy techniques (not only by getting bounds for harmonics of the numerical solution as in most papers).
In particular, the last point allows us to prove rigorously the 4th order error estimate in the strong energy norm for smooth solutions.
\par Moreover, enlarging of most schemes to the case of the wave equation with the variable coefficient $\rho(x)$ in front of $\partial_t^2u$ is simple, and there exists some connection to \cite{BTT18,STT19}.
Also the main schemes are rather easily generalized for non-uniform rectangular meshes in space and time; we apply averaging technique to both aims.
Concerning compact schemes on non-uniform meshes for other (1D in space) equations, in particular, see \cite{ChS18,JIS84,Z15,ZC18}.

\par In our 1D numerical experiments, we first concentrate on demonstrating the sensitive dependence of the error orders in the mesh $L^2$, uniform and strong energy norms on the weak smoothness order of the both initial functions and the weak dominating mixed smoothness order of the free term.
The cases of the delta-shaped, discontinuious or with discontinuos derivatives data are covered.
The higher-order practical error behavior is shown compared to standard 2nd approximation order schemes \cite{Z94,ZKMMA2018} thus confirming the essential advantages of 4th order schemes over them in the non-smooth case as well.
Second, we present numerical results in the case of non-uniform spatial meshes with various node distribution functions (for the smooth data).

The paper is organized as follows.
Auxiliary Section \ref{general3level} contains results on stability of general symmetric three-level method with a weight for hyperbolic equations in the strong and weak energy norms that we need to apply.
The main Section \ref{numerovschemes} is devoted to construction and analysis of the compact 4th order finite-difference schemes.
In Section \ref{nonunif_mesh}, the main compact schemes are generalized to the case of non-uniform rectangular meshes.
The results of these sections have been received by A. Zlotnik.
Section \ref{numerexperiments} contains results of numerical experiments have been accomplished by O. Kireeva.

\section{General symmetric three-level method for second order hyperbolic equations and its stability theorem}
\label{general3level}
\setcounter{equation}{0}
\setcounter{lemma}{0}
\setcounter{theorem}{0}

\par Let $H_h$ be a family of Euclidean spaces endowed with an inner product $(\cdot,\cdot)_h$ and the corresponding norm $\|\cdot\|_h$, where $h$ is the parameter (related to a spatial discretization).
Let linear operators $B_h$ and $A_h$ act in $H_h$ and have the properties $B_h=B_h^*>0$ and $A_h=A_h^*>0$.
Define the norms $\|w\|_{B_h}=(B_hw,w)_h^{1/2}$ and $\|w\|_{A_h}=(A_hw,w)_h^{1/2}$ in $H_h$ generated by them.
\par We assume that they are related by the following inequality
\begin{gather}
\|w\|_{A_h}\leq\alpha_h\|w\|_{B_h}\ \ \ \forall w\in H_h \ \ \Leftrightarrow\ \ A_h\leq\alpha_h^2B_h.
\label{ahbh}
\end{gather}
For methods of numerical solving 2nd order elliptic equations, usually $\alpha_h=c_0/h_{\min}$, where $h_{\min}$ is a minimal size of the spatial discretization.
\par We introduce the uniform mesh $\overline\omega_{h_t}=\{t_m=mh_t\}_{m=0}^M$
on a segment $[0,T]$, with the step $h_t=T/M>0$ and $M\geq 2$.
Let $\omega_{h_t}=\{t_m\}_{m=1}^{M-1}$.
We introduce the mesh averages and difference operators
\[
 \bar{s}_ty=\frac{\check{y}+y}{2},\,\
 s_ty=\frac{y+\hat{y}}{2},\,\
 \bar{\delta}_ty=\frac{y-\check{y}}{h_t},\,\
 \delta_ty=\frac{\hat{y}-y}{h_t},\,\
 \mathring{\delta}_ty=\frac{\hat{y}-\check{y}}{2h_t}
\]
and $\Lambda_ty=\delta_t\bar{\delta}_ty=\frac{\hat{y}-2y+\check{y}}{h_t^2}$
with $y^m=y(t_m)$, $\check{y}^{m}=y^{m-1}$ and $\hat{y}^{m}=y^{m+1}$,
as well as the summation operator with the variable upper limit
$I_{h_t}^my=h_t\sum_{l=1}^m y^l$ for $1\leq m\leq M$ and $I_{h_t}^0y=0$.

\par We consider a general symmetric three-level in $t$ method with a weight $\sigma$:
\begin{gather}
\big(B_h+\sigma h_t^2A_h\big)\Lambda_tv+A_hv=f\ \ \text{in}\ \ H_h\ \ \text{on}\ \ \omega_{h_t},
\label{3level sch1}\\[1mm]
 \big(B_h+\sigma h_t^2A_h\big)\delta_tv^0+\half h_tA_hv^0=u_1+\half h_tf^0\ \ \text{in}\ \ H_h
\label{3level sch2}
\end{gather}
where $v$: $\overline\omega_{h_t}\to H_h$ is the sought function and the functions $v^0,u_1\in H_h$ and $f$: $\{t_m\}_{m=0}^{M-1}\to H_h$ are given; we omit their dependence on $h$ for brevity.
Note that the parameter $\sigma$ can depend on $\*h:=(h,h_t)$.
Recall that linear algebraic systems with one and the same operator $B_h+\sigma h_t^2A_h$ has to be solved at time levels $t_m$
to find the solution $v^{m+1}$, $0\leq m\leq M-1$.
Note that \eqref{3level sch2} can be rewritten in the form closer to \eqref{3level sch1}:
$\frac{\big(B_h+\sigma h_t^2A_h\big)\delta_tv^0-u_1}{0.5h_t}+A_hv^0=f^0$.
\par Let the following conditions related to $\sigma$ hold: either $\sigma\geq\frac14$ and $\ve_0=1$, or
\begin{gather}
 \sigma<\tfrac14,\ \
 (\tfrac14-\sigma)h_t^2\alpha_h^2\leq 1-\ve_0^2\ \
 \text{for some}\ \ 0<\ve_0<1.
\label{stabcond}
\end{gather}
Then one can introduce the following $\sigma$- and $h_t$-dependent norm in $H_h$ and bound it from below:
\begin{equation}
\ve_0\|w\|_{B_h}\leq\|w\|_{0,\*h}:=\big[\|w\|_{B_h}^2+(\sigma-\tfrac14)h_t^2\|w\|_{A_h}^2\big]^{1/2}\ \
\forall w\in H_h.
\label{ve_0ineq}
\end{equation}
Obviously, for $\sigma\leq\tfrac14$, one also has $\|w\|_{0,\*h}\leq\|w\|_{B_h}$, and then the norms $\|\cdot\|_{0,\*h}$ and $\|\cdot\|_{B_h}$ are equivalent uniformly in $\*h$.

\par We present the stability theorem for method \eqref{3level sch1}-\eqref{3level sch2} with respect to the initial data $v^0$ and $u_1$ and the free term $f$ in the strong (standard) and weak energy mesh norms.

Define the norm
$\|y\|_{L_{h_t}^1(H_h)}=\tfrac14 h_t\|y^0\|_h+I_{h_t}^{M-1}\|y\|_h$ for $y$: $\{t_m\}_{m=0}^{M-1}\to H_h$.
\begin{theorem}
\label{theo:1}
For the solution to method \eqref{3level sch1}-\eqref{3level sch2}, the following bounds hold:
\par (1) in the strong energy norm
\begin{gather}
\max_{1\leq m\leq M}
\big[\|\bar{\delta}_tv^m\|_{B_h}^2+(\sigma-\tfrac14)h_t^2\|\bar{\delta}_tv^m\|_{A_h}^2
+\|\bar{s}_tv^m\|_{A_h}^2\big]^{1/2}
\nonumber\\[1mm]
\leq\big(\|v^0\|_{A_h}^2+\ve_0^{-2}\|B_h^{-1/2}u_1\|_h^2\big)^{1/2}
+2\ve_0^{-1}\|B_h^{-1/2}f\|_{L_{h_t}^1(H_h)};
\label{energy est1}
\end{gather}
one can replace the $f$-term
with
$2I_{h_t}^{M-1}\|A_h^{-1/2}\bar{\delta}_tf\|_h+3\max\limits_{0\leq m\leq M-1}\|A_h^{-1/2}f^m\|_h$;

\par (2) in the weak energy norm
\begin{gather}
\max_{0\leq m\leq M}
\max\Big\{\big[\|v^m\|_{B_h}^2+(\sigma-\tfrac14)h_t^2\|v^m\|_{A_h}^2\big]^{1/2},\,\|I_{h_t}^m\bar{s}_tv\|_{A_h}\Big\}
\nonumber\\[1mm]
 \leq\big[\|v^0\|_{B_h}^2+(\sigma-\tfrac14)h_t^2\|v^0\|_{A_h}^2\big]^{1/2}
 +2\|A_h^{-1/2}u_1\|_h
 +2\|A_h^{-1/2}f\|_{L_{h_t}^1(H_h)}.
\label{energy est2}
\end{gather}
For $f=\delta_tg$, one can replace $2\|A_h^{-1/2}f\|_{L_{h_t}^1(H_h)}$ with $\tfrac{2}{\ve_0}I_{h_t}^{M}\|B_{h}^{-1/2}\big(g-s_tg^0\big)\|_h$.
\end{theorem}
\begin{proof}
Similar bounds have recently been proved in \cite{ZCh20} for the method
\begin{gather*}
\big(\tau B_h+\sigma h_t^2A_h\big)\Lambda_tv+B_{1h}\mathring{\delta}_tv+A_hv=f\ \ \text{in}\ \ H_h\ \ \text{on}\ \ \omega_{h_t},
\\[1mm]
 \big(\tau B_h+\half h_t B_{1h}+\sigma h_t^2A_h\big)\delta_tv^0+\half h_tA_hv^0=u_1+\half h_tf^0\ \ \text{in}\ \ H_h
\end{gather*}
of a more general form, with the parameter $\tau>0$ and an operator $B_{1h}=B_{1h}^*>0$ acting in $H_h$.
In these bounds, one can take $\tau=1$ and easily see from their proofs that the bounds mainly remain valid for $B_{1h}=B_{1h}^*\geq 0$, in particular, $B_{1h}=0$ (the case considered here), up to the norm of $f$ standing in \eqref{energy est1} and the norm of $g-s_tg^0$ mentioned in Item 2.

\par To verify the validity of the bounds precisely with the norms of $f$ and $\tilde{g}:=g-s_tg^0$ indicated in this theorem, it suffices
to modify bounds for the following summands with $f$ in the strong energy equality in \cite[Theorem 1]{ZCh20}
\begin{gather*}
 \half h_t(f^0,\bar{\delta}_tv^1)_h+2I_{h_t}^{m-1}(f,\mathring{\delta}_tv)_h
 \leq\half h_t\|B_h^{-1/2}f^0\|_h\|\bar{\delta}_tv^1\|_{B_h}
\\[1mm]
 +2I_{h_t}^{M-1}\|B_h^{-1/2}f\|_h\max_{1\leq m\leq M}\|\bar{\delta}_tv^m\|_{B_h}
 \leq \tfrac{2}{\ve_0}
 \|B_h^{-1/2}f\|_{L_{h_t}^1(H_h)}\max_{1\leq m\leq M}\|\bar{\delta}_tv^m\|_{0,\*h}
\end{gather*}
and, setting $\check{I}_{h_t}^mf=I_{h_t}^{m-1}f$, in the weak energy equality in \cite[Theorem 2]{ZCh20}
\begin{gather*}
 2I_{h_t}^m\big(\half h_tf^0+\check{I}_{h_t}f,\bar{s}_tv\big)_h
 =2I_{h_t}^m\big(\tilde{g},\bar{s}_tv\big)_h
 \leq \tfrac{2}{\ve_0}I_{h_t}^M\|B_{h}^{-1/2}\tilde{g}\|_h\max_{1\leq m\leq M}\|\bar{s}_tv^m\|_{0,\*h},
\end{gather*}
for $1\leq m\leq M$,
and the relations $\mathring{\delta}_t=\half(\delta_t+\bar{\delta}_t)$ and \eqref{ve_0ineq}
have been applied.
\end{proof}

\par Clearly in fact the norm $\|\cdot\|_{0,\*h}$ stands on the left in \eqref{energy est1} and on both sides in \eqref{energy est2}.
Bounds of type \eqref{energy est1} with a stronger norm of $f$ can be found in \cite{S77}.

\par Below we also refer to the following stability result.
\begin{remark}
\label{rem1}
Under assumptions \eqref{stabcond} with $\ve_0=0$, instead of bound \eqref{energy est1} the following one holds
\begin{gather*}
 \max_{1\leq m\leq M}
 \big[\|\bar{\delta}_tv^m\|_{B_h}^2+(\sigma-\tfrac14)h_t^2\|\bar{\delta}_tv^m\|_{A_h}^2
 +\|\bar{s}_tv^m\|_{A_h}^2\big]^{1/2} \leq\big[\|v^0\|_{A_h}^2
\nonumber\\[1mm]
+\|(B_h+\sigma h_t^2A_h)^{-1/2}u_1\|_h^2\big]^{1/2}
 +2I_{h_t}^{M-1}\|A_h^{-1/2}\bar{\delta}_tf\|_h+3\max\limits_{0\leq m\leq M-1}\|A_h^{-1/2}f^m\|_h,
\end{gather*}
whereas bound \eqref{energy est2} remains valid (its proof does not change for $\ve_0\geq 0$).

To be convinced of the latter bound, it is necessary to transform and bound differently the terms with $v^0$ and $u_1$ in the case $f=0$ in the strong energy equality in \cite{ZCh20}.
Namely, using the formula $\bar{s}_tv^1=v^0+\half h_t\bar{\delta}_tv^1$ and equation \eqref{3level sch2} with $f^0=0$, we can set $C_h:=(B_h+\sigma h_t^2A_h)^{-1}$ and obtain
\begin{gather*}
 (A_hv^0,\bar{s}_tv^1)_h+(u_1,\bar{\delta}_tv^1)_h=\|v^0\|_{A_h}^2+\big(\half h_tA_hv^0+u_1,\bar{\delta}_tv^1\big)_h
=\|v^0\|_{A_h}^2
\nonumber\\[1mm]
 +\big(C_h\big(\half h_tA_hv^0+u_1\big),-\half h_tA_hv^0+u_1\big)_h
 =\|v^0\|_{A_h}^2+\|u_1\|_{C_h}^2-\|\half h_tA_hv^0\|_{C_h}^2
\label{inid1}
\end{gather*}
since $C_h=C_h^*>0$.
This implies the first bound of this Remark.
\par Notice that $B_h+\sigma h_t^2A_h\geq \ve_0B_h+\tfrac14 h_t^2A_h$ under the assumptions either $\sigma\geq\frac14$ and $\ve_0=1$, or \eqref{stabcond} with $0\leq\ve_0<1$ and, as a corollary,
$C_h\leq\ve_0^{-1}B_h^{-1}$ (for $\ve_0\neq 0$) and $C_h\leq 4h_t^{-2}A_h^{-1}$.

But, for $\ve_0=0$, the quantity $\|w\|_{0,\*h}$ could be (in general) only a semi-norm in $H_h$, and its lower bound by $\|w\|_{B_h}$ uniformly in $\*h$ is not valid any more.
\end{remark}
\par It is well-known that each of bounds \eqref{energy est1}-\eqref{energy est2} implies existence and uniqueness of the solution to method \eqref{3level sch1}-\eqref{3level sch2} for any given $v^0,u_1\in H_h$ and $f$: $\{t_m\}_{m=0}^{M-1}\to H_h$.
The same concerns finite-difference schemes below.

\section{Construction and properties of compact finite-difference schemes of the 4th order of approximation}
\label{numerovschemes}
\setcounter{equation}{0}
\setcounter{lemma}{0}
\setcounter{theorem}{0}

We consider the following IBVP with the nonhomogeneous Dirichlet boundary condition for the slightly generalized wave equation
\begin{gather}
 \partial_t^2u-a_i^2\partial_i^2u =f(x,t)\ \ \text{in}\ \ Q_T=\Omega\times (0,T);
\label{hyperb2eq}
\\[1mm]
 u|_{\Gamma_T}=g(x,t);\ \ u|_{t=0}=u_0(x),\ \ \partial_tu|_{t=0}=u_1(x),\ \ x\in\Omega.
\label{hyperb2ibc}
\end{gather}
Here $a_1>0,\ldots,a_n>0$ are constants, $x=(x_1,\ldots,x_n)$, $\Omega=(0,X_1)\times\ldots\times(0,X_n)$, $n\geq 1$,
$\partial\Omega$ is the boundary of $\Omega$ and $\Gamma_T=\partial\Omega\times (0,T)$ is the lateral surface of $Q_T$.
Hereafter the summation from 1 to $n$ over the repeated indices $i,j$ (and only over them) is assumed.
Below $\delta^{(ij)}$ is the Kronecker symbol.
\par Define the uniform rectangular mesh
$\bar{\omega}_h=\{x_{\*k}=(k_1h_1,\ldots,k_nh_n);\, 0\leq k_1\leq N_1,\ldots,0\leq k_n\leq N_n\}$ in $\bar{\Omega}$
with the steps $h_1=X_1/N_1,\ldots,h_n=X_n/N_n$, $h=(h_1,\ldots,h_n)$ and $\*k=(k_1,\ldots,k_n)$.
Let $\omega_h=\{x_{\*k};\, 1\leq k_1\leq N_1-1,\ldots,1\leq k_n\leq N_n-1\}$ and $\partial\omega_h=\bar{\omega}_h\backslash\omega_h$ be the internal part and boundary of $\bar{\omega}_h$.
Define the meshes $\omega_{\*h}:=\omega_h\times\omega_{h_t}$ in $Q_T$ and $\partial\omega_{\*h}=\partial\omega_h\times\{t_m\}_{m=1}^M$ on $\bar{\Gamma}_T$.

\par We introduce the well-known difference operators
$ (\Lambda_lw)_{\*k}=\tfrac{1}{h_l^2}(w_{\*k+\*e_l}-2w_{\*k}+w_{\*k-\*e_l})$, $l=1,\ldots,n$, on $\omega_{h}$,
where $w_{\*k}=w(x_{\*k})$ and $\*e_1,\ldots,\*e_n$ is the standard coordinate basis in $\mathbb{R}^n$.
\par Let below $H_h$ be the space of functions defined on $\bar{\omega}_h$, equal 0 on $\partial\omega_h$ and endowed with the inner product
$ (v,w)_h=h_1\ldots h_n\sum\nolimits_{x_{\*k}\in\omega_h}v_{\*k}w_{\*k}$
and the norm $\|w\|_h=(w,w)_h^{1/2}$.
\begin{lemma}
\label{lem:psi}
For the sufficiently smooth in $\bar{Q}_T$ solution $u$ to equation \eqref{hyperb2eq},
the following formula holds
\begin{gather}
 \big(s_N-\onetwelve h_t^2a_i^2\Lambda_i\big)\Lambda_tu-a_j^2s_{N\hat{j}}\Lambda_ju-f_N=O(|\*h|^4)\ \
 \text{on}\ \ \omega_{\*h},
\label{approxerrN}
\end{gather}
where
\[
 s_N:=I+\onetwelve h_i^2\Lambda_i,\,\
 s_{N\hat{j}}:=I+(1-\delta^{(ij)})\onetwelve h_i^2\Lambda_i,\,\
 f_N:=f+\onetwelve h_t^2\Lambda_tf+\onetwelve h_i^2\Lambda_if,
\]
and $I$ is the identity operator.
Note that $s_{N\hat{j}}=I$ for $n=1$.
\end{lemma}
\begin{proof}
\par We give two different proofs.

\par 1. The first one follows to the classical Numerov approach.
We take the simplest explicit three-level discretization of equation \eqref{hyperb2eq} having the form
\[
 \Lambda_tv-a_i^2\Lambda_iv=f\ \ \text{on}\ \ \omega_{\*h}
\]
(the particular case of equation \eqref{3level sch1} for $B_h=I$, $A_h=-a_i^2\Lambda_iv$ and $\sigma=0$)
and, under the assumption of sufficient smoothness of $u$, select the leading term of its approximation error
$\psi_e:=\Lambda_tu-a_i^2\Lambda_iu-f$:
\begin{equation}
\psi_e
=\Lambda_tu-\partial_t^2u-a_i^2(\Lambda_iu-\partial_i^2u)
=\onetwelve h_t^2\partial_t^4u-\onetwelve h_i^2a_i^2\partial_i^4u+O(|\*h|^4).
\label{apprerror1}
\end{equation}
We express the derivatives $\partial_t^4u$ and $\partial_k^4u$ in terms of mixed derivatives by differentiating equation \eqref{hyperb2eq}:
\begin{equation}
 \partial_t^4u=a_i^2\partial_i^2\partial_t^2u+\partial_t^2f,\ \
 a_k^2\partial_k^4u=\partial_k^2\partial_t^2u-(1-\delta^{(kj)})a_j^2\partial_k^2\partial_j^2u-\partial_k^2f.
\label{cons wave eq}
\end{equation}
Then formula \eqref{apprerror1} takes the form
\[
\psi_e
=\tfrac{h_t^2}{12}a_i^2\partial_i^2\partial_t^2u-\tfrac{h_i^2}{12}\partial_i^2\partial_t^2u
+\tfrac{h_i^2}{12}(1-\delta^{(ij)})a_j^2\partial_i^2\partial_j^2u
+\tfrac{h_t^2}{12}\partial_t^2f+\tfrac{h_i^2}{12}\partial_i^2f+O(|\*h|^4).
\]
Here all the 2nd order derivatives can be replaced by the corresponding symmetric three-point difference discretizations preserving the order of the remainder:
\[
\psi_e
=\tfrac{h_t^2}{12}a_i^2\Lambda_i\Lambda_tu-\tfrac{h_i^2}{12}\Lambda_i\Lambda_tu
+\tfrac{h_i^2}{12}(1-\delta^{(ij)})a_j^2\Lambda_i\Lambda_ju
+\tfrac{h_t^2}{12}\Lambda_tf+\tfrac{h_i^2}{12}\Lambda_if+O(|\*h|^4).
\]
Recalling the definition of $\psi_e$ in \eqref{apprerror1}, we can rewrite the last formula as \eqref{approxerrN}.

\smallskip\par 2. The second proof is based on averaging of equation \eqref{hyperb2eq} related to the polylinear finite elements.
We define the well-known average in the variable $x_k$ related to the linear finite elements
\[
 (q_kw)(x_k)=\tfrac{1}{h_k}\int_{-h_k}^{h_k}w(x_k+\xi)\big(1-\tfrac{|\xi|}{h_k}\big)\,d\xi.
\]
For a function $w(x_k)$ smooth on $[0,X_k]$, the following formulas hold
\begin{gather}
 q_k\partial_k^2w=\Lambda_kw,\ \
\label{qklambdak}
\\[1mm]
 q_kw=w+\onetwelve h_k^2\partial_k^2w+q_k\rho_{k4}(\partial_k^4w)=w+\onetwelve h_k^2\Lambda_kw+\tilde{\rho}_{k4}(\partial_k^4w),
\nonumber\\[1mm]
\hspace{-8pt} |q_k\rho_{ks}(\partial_k^sw)|\leq c_sh_k^s\|\partial_k^sw\|_{C(I_{kl})},\, s=2,4,\
 |\tilde{\rho}_{k4}(\partial_k^4w)|\leq \tilde{c}_4h_k^4\|\partial_k^4w\|_{C(I_{kl})}
\label{resid_bound}
\end{gather}
and
$q_kw=w+q_k\rho_{k2}(\partial_k^2w)$
at the nodes $x_k=x_{kl}:=lh_k$, $1\leq l\leq N_k-1$, with $I_{kl}:=[x_{k(l-1)},x_{k(l+1)}]$.
The first formula is checked by integrating by parts and other formulas hold owing to
the Taylor formula at $x_{kl}$ with the residual in the integral form
\begin{gather}
 \rho_{ks}(w)(x_k)=\tfrac{1}{(s-1)!}\int_{x_{kl}}^{x_k}w(\xi)(x_k-\xi)^{s-1}\,d\xi,
\label{taylor_residual}
\end{gather}
for $s=2,4$,
together with $\tfrac{1}{h_k}\int_{-h_k}^{h_k}\half \xi^2\big(1-\tfrac{|\xi|}{h_k}\big)\,d\xi=\onetwelve h_k^2$.
The respective formulas hold for the averaging operator $q_t$ in the variable $t=x_{n+1}$ as well
(since one can set $X_{n+1}=T$ and $h_{n+1}=h_t$).

We apply the operator $\bar{q}q_t$ with $\bar{q}:=q_1\ldots q_n$ to
equation
\eqref{hyperb2eq} at the nodes of $\omega_{\*h}$ and get
\begin{gather}
 \bar{q}\Lambda_tu-a_i^2 \bar{q}_{\hat{i}}q_t\Lambda_iu=\bar{q}q_tf\ \ \text{with}\ \
 \bar{q}_{\hat{i}}:=\prod_{1\leq k\leq n,\, k\neq i}q_k.
\label{avereq}
\end{gather}
The multiple application of the above formulas for the averages leads to
\begin{gather*}
 \Lambda_tu+\onetwelve h_i^2\Lambda_i\Lambda_tu
 -a_i^2\big[\Lambda_i^2u+(1-\delta^{(ij)})\onetwelve h_j^2\Lambda_j\Lambda_iu+\onetwelve h_t^2\Lambda_i\Lambda_tu\big]
\\[1mm]
 =f+\onetwelve h_i^2\Lambda_if+\onetwelve h_t^2\Lambda_tf+O(|\*h|^4),
\end{gather*}
and thus formula \eqref{approxerrN} is derived once again.
\end{proof}
\begin{remark}
For the first order in time parabolic equation or TDSE, one should apply the simpler averaging
$q_ty^m=\frac{1}{h_t}\int_{t_{m-1}}^{t_m}y(t)\,dt$ in time to derive two-level higher-order compact schemes.
\end{remark}
\par Formula \eqref{approxerrN} means that the discretization of equation \eqref{hyperb2eq} of the form
\begin{equation}
 \big(s_N-\onetwelve h_t^2a_i^2\Lambda_i\big)\Lambda_tv-a_i^2s_{N\hat{i}}\Lambda_iv=f_N\ \ \text{on}\ \ \omega_{\*h}
\label{num0eq}
\end{equation}
has the approximation error of the order $O(|\*h|^4)$.
\par Notice that the coefficients of formulas
\begin{gather*}
 y+\onetwelve h_t^2\Lambda_ty=\onetwelve(\hat{y}+10f+\check{y}),\,\
 \onetwelve h_i^2\Lambda_iw_{\*k}=\onetwelve\delta^{(ii)}(w_{\*k-\*e_i}+w_{\*k+\*e_i})-\tfrac{n}{6}w_{\*k}
\end{gather*}
respectively on $\omega_{h_t}$ and $\omega_h$ are independent of $\*h$.
\par For discretization \eqref{num0eq}, we consider the corresponding equation at $t_0=0$
\begin{equation}
 \big(s_N-\onetwelve h_t^2a_i^2\Lambda_i\big)\delta_tv^0-\half h_ta_i^2s_{N\hat{i}}\Lambda_iv^0=u_{1N}+\half h_tf_N^0\ \ \text{on}\ \ \omega_h,
\label{num0ic}
\end{equation}
cp. \eqref{3level sch1}-\eqref{3level sch2}, and find out for which $u_{1N}$ and $f_N^0$ its approximation error also has the order $O(|\*h|^4)$.
Let $0<\bar{h}_t\leq T$ and $h_t\leq\bar{h}_t$.
\begin{lemma}
\label{lem:psi0}
For the sufficiently smooth in $\bar{Q}_{\bar{h}_t}$ solution $u$ to equation \eqref{hyperb2eq} satisfying the initial conditions from \eqref{hyperb2ibc}, under the choice
\begin{gather}
 u_{1N}=\big(s_N+\onetwelve h_t^2a_i^2\Lambda_i\big)u_1,\,\
\label{u1N}\\[1mm]
 f_N^0=f_{dh_t}^{(0)}+\onetwelve h_i^2\Lambda_if^0,\ \
 f_{dh_t}^{(0)}=f_{d}^{(0)}+O(h_t^3)
\label{tf0N}
\end{gather}
on $\omega_h$, where $f_{d}^{(0)}:=f_0+\tfrac13 h_t(\partial_tf)_0+\onetwelve h_t^2(\partial_t^2f)_0$
with $y_0:=y|_{t=0}$,
the approximation error of equation \eqref{num0ic} satisfies the following formula
\begin{gather}
\psi_e^0:=\big(s_N-\tfrac{h_t^2}{12}a_i^2\Lambda_i\big)(\delta_tu)^0-\tfrac{h_t}{2}a_i^2s_{N\hat{i}}\Lambda_iu_0-u_{1N}-\tfrac{h_t}{2}f_N^0
 =O(|\*h|^4).
\label{psie0}
\end{gather}
Notice
that $f_{d}^{(0)}$ is not
the term $f_0+\onetwelve h_t^2(\partial_t^2f)_0$ of type approximated above.
\end{lemma}
\begin{proof}
Let $0\leq t\leq\bar{h}_t$.
Once again we give two proofs.

1. Using Taylor's formula in $t$ and grouping separately terms with the time derivatives of odd and even orders, we obtain
\begin{gather*}
 \psi_e^0=\big(s_N-\onetwelve h_t^2a_i^2\Lambda_i\big)(\partial_tu)_0+\tfrac16 h_t^2(\partial_t^3u)_0
 +\half h_t\big[\big(s_N-\onetwelve h_t^2a_i^2\Lambda_i\big)(\partial_t^2u)_0
\\[1mm]
 +\onetwelve h_t^2(\partial_t^4u)_0
 -a_i^2s_{N\hat{i}}\Lambda_iu_0\big]-u_{1N}-\half h_tf_N^0+O(|\*h|^4).
\end{gather*}
In virtue of equation \eqref{hyperb2eq} we have
\[
 \partial_t^3u=a_i^2\partial_i^2\partial_tu+\partial_tf=a_i^2\Lambda_i^2\partial_tu+\partial_tf+O(|h|^2).
\]
Moreover, $(\partial_tu)_0=u_1$, therefore we find
\begin{gather}
 \big(s_N-\onetwelve h_t^2a_i^2\Lambda_i\big)(\partial_tu)_0+\tfrac16 h_t^2(\partial_t^3u)_0
\nonumber\\[1mm]
 =\big(s_N+\onetwelve h_t^2a_i^2\Lambda_i\big)u_1+\tfrac16 h_t^2(\partial_tf)_0+O(|\*h|^4).
\label{psie0_1}
\end{gather}
Next, the first formula \eqref{cons wave eq} implies $\partial_t^4u=a_i^2\Lambda_i\partial_t^2u+\partial_t^2f+O(|h|^2)$ and thus
\begin{gather*}
 \big(s_N-\onetwelve h_t^2a_i^2\Lambda_i\big)(\partial_t^2u)_0
 +\onetwelve h_t^2(\partial_t^4u)_0
 =s_N(\partial_t^2u)_0+\onetwelve h_t^2(\partial_t^2f)_0+O(|\*h|^4).
\end{gather*}
Using \eqref{hyperb2eq} for $t=0$ and the formula $s_N=s_{N\hat{k}}+\onetwelve h_k^2\Lambda_k$, we also have
\begin{gather*}
 s_N(\partial_t^2u)_0-a_i^2s_{N\hat{i}}\Lambda_iu_0
 =s_N(a_i^2\partial_i^2u_0+f_0)-a_i^2s_{N\hat{i}}\Lambda_iu_0
\\[1mm]
  =a_i^2s_{N\hat{i}}(\partial_i^2u_0-\Lambda_iu_0)
 +\onetwelve h_i^2a_i^2\Lambda_i\partial_i^2u_0+s_Nf_0
\\[1mm]
 =a_i^2s_{N\hat{i}}\big(-\onetwelve h_i^2\partial_i^4u_0\big)+\onetwelve h_i^2a_i^2\Lambda_i\partial_i^2u_0+s_Nf_0+O(|h|^4)
\\[1mm]
 =\onetwelve h_i^2a_i^2(\Lambda_i\partial_i^2u_0-\partial_i^4u_0)+s_Nf_0+O(|h|^4)=s_Nf_0+O(|h|^4).
\end{gather*}
Therefore we have proved the formula
\[
 \big(s_N-\onetwelve h_t^2a_i^2\Lambda_i\big)(\partial_t^2u)_0
 +\onetwelve h_t^2(\partial_t^4u)_0
 -a_i^2s_{N\hat{i}}\Lambda_iu_0
 =s_Nf_0+\onetwelve h_t^2(\partial_t^2f)_0+O(|h|^4).
\]
This formula and
\eqref{psie0_1} under choice \eqref{u1N}-\eqref{tf0N} lead to formula \eqref{psie0}.

\par 2. Again the second proof is based on averaging of equation \eqref{hyperb2eq}.
We define the related one-sided average in $t$ over $(0,h_t)$
\begin{gather}
 q_ty^0=\tfrac{2}{h_t}\int_0^{h_t}y(t)\big(1-\tfrac{t}{h_t}\big)\,dt
\label{qty0}
\end{gather}
and apply $\tfrac{h_t}{2}q_t(\cdot)^0$ to
\eqref{hyperb2eq}.
Since $\tfrac{h_t}{2}(q_t\partial_tu)^0=(\delta_tu)^0-(\partial_tu)_0$, we get
\begin{gather}
 \bar{q}(\delta_tu)^0-\tfrac{h_t}{2}a_i^2\bar{q}_{\hat{i}}\Lambda_iq_tu^0=\bar{q}u_1+\tfrac{h_t}{2}\bar{q}q_tf^0.
\label{avereq_0}
\end{gather}
Using Taylor's formula
at $t=0$ and calculating the arising integrals, we find
\begin{gather}
 \tfrac{h_t}{2}q_tf^0=\tfrac{h_t}{2}f_0+\tfrac{h_t^2}{6}(\partial_tf)_0+\tfrac{h_t^3}{24}(\partial_t^2f)_0+O(h_t^4)
 =\tfrac{h_t}{2}f_d^{(0)}+O(h_t^4).
\label{expan_f}
\end{gather}
Here we omit the integral representations for $O(h_t^4)$-terms for brevity.
As in the proof of Lemma \ref{lem:psi} and owing to the last expansion, we have $\bar{q}(\delta_tu)^0=s_N(\delta_tu)^0+O(|h|^4)$ and \begin{gather}
 \bar{q}u_1=s_Nu_1+O(|h|^4),\ \
 \tfrac{h_t}{2}q_t\bar{q}f^0=\tfrac{h_t}{2}f_d^{(0)}+\tfrac{1}{12}h_i^2\Lambda_if_0+O(|\*h|^4).
\label{aver_prop}
\end{gather}
\par Also owing to Taylor's formula in $t$ at $t=0$ we can write down
\[
 u(\cdot,t)=u_0+tu_1+\tfrac{t^2}{h_t}((\delta_tu)^0-u_1)+O(t^3).
\]
Thus similarly first to \eqref{expan_f} and second to \eqref{aver_prop} we obtain
\begin{gather*}
 \tfrac{h_t}{2}a_i^2\bar{q}_{\hat{i}}\Lambda_iq_tu^0
 =\tfrac{h_t}{2}a_i^2\bar{q}_{\hat{i}}\Lambda_iu_0+\tfrac{h_t^2}{6}a_i^2\bar{q}_{\hat{i}}\Lambda_iu_1
 +\tfrac{h_t^2}{12}a_i^2\bar{q}_{\hat{i}}\Lambda_i((\delta_tu)^0-u_1)+O(h_t^4)
\\[1mm]
 =\tfrac{h_t}{2}a_i^2s_{N\hat{i}}\Lambda_iu_0
 +\tfrac{h_t^2}{12}a_i^2\Lambda_iu_1
 +\tfrac{h_t^2}{12}a_i^2s_{N\hat{i}}\Lambda_i(\delta_tu)^0+O(|\*h|^4).
\end{gather*}
Inserting all the derived formulas into \eqref{avereq_0}, we again obtain the desired result.
\end{proof}
\begin{remark}
If $f$ is sufficiently smooth in $t$ in $\bar{Q}_{\bar{h}_t}$, then the property $f_{dh_t}^{(0)}=f_{d}^{(0)}+O(h_t^3)$ (see \eqref{tf0N}) holds for the following three- and two-level approximations
\begin{gather*}
f_{dh_t}^{(0)}=\tfrac{7}{12}f^0+\half f^1-\onetwelve f^2,\ \
f_{dh_t}^{(0)}=\tfrac13f^0+\tfrac23f^{1/2}\ \ \text{with}\ \  f^{1/2}:=f|_{t=h_t/2}.
\label{ftd02}
\end{gather*}
One can easily check this using the Taylor formula in $t$ at $t=0$.

If $f$ is sufficiently smooth in $t$ in $\bar{\Omega}\times [-\bar{h}_t,\bar{h}_t]$, then clearly the same property
holds for the one more three-level approximation
\begin{gather*}
 f_{dh_t}^{(0)}=f^0+\tfrac13 h_t\mathring{\delta}_tf^0+\onetwelve h_t^2\Lambda_tf^0
 =-\onetwelve f^{-1}+\tfrac56f^0+\tfrac14 f^1\ \text{with}\ f^{-1}:=f|_{t=-h_t}.
\end{gather*}
\end{remark}
\begin{remark}
\label{rem:nonsmooth_f}
Below we consider the case of non-smooth $f$.
Namely the above second proofs of Lemmas \ref{lem:psi}-\ref{lem:psi0} clarify that then $f_N^m$ should be replaced with
$\bar{q}q_tf^m$, $0\leq m\leq M-1$, according to \eqref{avereq} and \eqref{avereq_0} and identically to the polylinear FEM with the weight \cite{Z94}, or with some its suitable approximation.
\end{remark}

\par In the simplest case $n=1$,
equations \eqref{num0eq}-\eqref{num0ic} supplemented with the boundary condition take the following form
\begin{gather}
 \big[I+\tfrac{1}{12}(h_1^2-a_1^2h_t^2)\Lambda_1\big]\Lambda_tv-a_1^2\Lambda_1v=f_N,
\label{num1eq1d}\\[1mm]
 v|_{\partial\omega_{\*h}}=g,\ \
 \big[I+\tfrac{1}{12}(h_1^2-a_1^2h_t^2)\Lambda_1\big]\delta_tv^0-\half h_ta_1^2\Lambda_1v^0=u_{1N}+\half h_tf_N^0,
\label{num1ic1d}
\end{gather}
where equations are valid respectively on $\omega_{\*h}$ and $\omega_h$.
Hereafter we assume that the function $v^0$ is given on $\bar{\omega}_h$ and take the general nonhomogeneous Dirichlet boundary condition.
This scheme can be interpreted as the particular case of scheme \eqref{3level sch1}-\eqref{3level sch2} with the operators $B_h=I$ and $A_h=-a_1^2\Lambda_1$ and the weight
$\sigma=\sigma(\*h)=\tfrac{1}{12}\big(1-\tfrac{h_1^2}{a_1^2h_t^2}\big)$
(a similar choice of $\sigma$ was used in \cite{S77} in the 1D parabolic case) or
the bilinear finite element method \cite{Z94}
with $B_h=I+\tfrac{1}{6}h_1^2\Lambda_1$, $A_h=-a_1^2\Lambda_1$ and
$\sigma=\sigma(\*h)=\tfrac{1}{12}\big(1+\tfrac{h_1^2}{a_1^2h_t^2}\big)$
(though the right-hand sides of the equations are not the same; but see also Remark \ref{rem:nonsmooth_f}).

\par But for $n\geq 2$ the above constructed equations \eqref{num0eq}-\eqref{num0ic} are not of type
\eqref{3level sch1}-\eqref{3level sch2}.
Therefore we replace them with the following one
\begin{gather}
 \big(s_N+\onetwelve h_t^2A_N\big)\Lambda_tv+A_Nv=f_N\ \ \text{on}\ \ \omega_{\*h},
\label{num1eq}\\[1mm]
  v|_{\partial\omega_{\*h}}=g,\ \
  (s_N+\onetwelve h_t^2A_N)\delta_tv^0+\half h_tA_Nv^0=u_{1N}+\half h_tf_N^0\ \ \text{on}\ \ \omega_h,
\label{num1ic}
\end{gather}
where $A_N:=-a_i^2s_{N\hat{i}}\Lambda_i$, that corresponds to the case $B_h=s_N$, $A_h=A_N$ and $\sigma=\onetwelve$.
Since $A_N+a_i^2\Lambda_i=a_i^2(I-s_{N\hat{i}})\Lambda_i$, we have
$h_t^2(A_N+a_i^2\Lambda_i)\Lambda_tu=O(|\*h|^4)$ and $h_t^2(A_N+a_i^2\Lambda_i)(\delta_tu)^0=O(|\*h|^4)$ for a function $u$ sufficiently smooth in $\bar{Q}_T$, and thus the approximation errors of the both equations of this scheme are also of the order $O(|\*h|^4)$.

\par But the latter scheme fails for $n\geq 3$ similarly to \cite{DZR15} in the case of the TDSE.
The point is that $s_N$ should approximate $I$ adequately, but for the minimal and maximal eigenvalues of
$s_N<I$ as the operator in $H_h$ we have
\[
\lambda_{\min}(s_N)=1-\tfrac{1}{3}\sin^2\tfrac{\pi (N_i-1)}{2N_i}<\lambda_{\max}(s_N)<1.
\]
Therefore $\lambda_{\min}(s_N)>1-\tfrac{n}{3}$ and $\lambda_{\min}(s_N)=1-\tfrac{n}{3}+O\big(\delta^{(ii)}\frac{1}{N_i^2}\big)$ that is suitable for $n=1,2$,
but
$s_N$ becomes almost singular for $n=3$ and even $\lambda_{\min}(s_N)<0$ (i.e., $s_N$ is not positive definite any more) for $n\geq 4$, for small $|h|$.

Thus for $n=3$ it is of sense to replace the last scheme with the scheme
\begin{gather}
 \big(\bar{s}_N+\onetwelve h_t^2A_N\big)\Lambda_tv+A_Nv=f_N\ \ \text{on}\ \ \omega_{\*h},
\label{num2eq}\\[1mm]
 v|_{\partial\omega_{\*h}}=g,\ \
 (\bar{s}_N+\onetwelve h_t^2A_N)\delta_tv^0+\half h_tA_Nv^0=u_{1N}+\half h_tf_N^0\ \ \text{on}\ \ \omega_h.
\label{num2ic}
\end{gather}
Moreover, for any $n\geq 1$ we can use the following scheme
\begin{gather}
 \big(\bar{s}_N+\onetwelve h_t^2\bar{A}_N\big)\Lambda_tv+\bar{A}_Nv=f_N\ \ \text{on}\ \ \omega_{\*h},
\label{num3eq}\\[1mm]
 v|_{\partial\omega_{\*h}}=g,\ \
 \big(\bar{s}_N+\onetwelve h_t^2\bar{A}_N\big)\delta_tv^0+\half h_t\bar{A}_Nv^0=u_{1N}+\half h_tf_N^0\ \ \text{on}\ \ \omega_h
\label{num3ic}
\end{gather}
(cp. \cite{DZR15} in the case of the TDSE);
for $n=1$ it coincides with \eqref{num1eq1d}-\eqref{num1ic1d}.
Here the operators
\[
 \bar{s}_N:=\prod_{k=1}^ns_{kN},\,\ \bar{s}_{N\hat{l}}:=\prod_{1\leq k\leq n,\,k\neq l}s_{kN},\,\
 s_{kN}:=I+\onetwelve h_k^2\Lambda_k,\,\
 \bar{A}_N:=-a_i^2\bar{s}_{N\hat{i}}\Lambda_i
\]
are used, with $\bar{s}_{N\hat{l}}=I$ for $n=1$.
The operator $\bar{s}_N$ is the splitting version of $s_N$, and $\bar{s}_{N\hat{l}}$ is the $(n-1)$-dimensional case of $\bar{s}_N$.
Clearly $\bar{A}_N=A_N$ for $n=1,2$.
Herewith for the minimal and maximal eigenvalues of $\bar{s}_N<I$ as the operator in $H_h$ we have
\[
 (\tfrac23)^n<\lambda_{\min}(\bar{s}_N)=\prod_{k=1}^n 1-\tfrac{1}{3}\sin^2\tfrac{\pi (N_k-1)}{2N_k}<\lambda_{\max}(\bar{s}_N)<1.
\]
Moreover, the following relation between $\bar{s}_N$ and $s_N$ holds
\begin{gather}
 \bar{s}_N=s_N+\sum_{k=2}^n\sum_{1\leq i_1<\ldots<i_k\leq n}
 \onetwelve h_{i_1}^2\Lambda_{i_1}\ldots\onetwelve h_{i_k}^2\Lambda_{i_k}.
\label{bsNsN}
\end{gather}

\par In virtue of the last formula we have $(\bar{s}_N-s_N)\Lambda_tu=O(|h|^4)$ and $(\bar{s}_N-s_N)(\delta_tu)^0=O(|h|^4)$ for a function $u$ sufficiently smooth in $\bar{Q}_T$, thus the approximation errors of the both equations of scheme \eqref{num2eq}-\eqref{num2ic} still have the order $O(|\*h|^4)$ as for the previous scheme \eqref{num1eq}-\eqref{num1ic}.

\par Since $\bar{A}_N-A_N=-a_i^2(\bar{s}_{N\hat{i}}-s_{N\hat{i}})\Lambda_i$, in virtue of \eqref{bsNsN} we have
$(\bar{A}_N-A_N)y=O(|h|^4)$
for $y=\Lambda_tu, u, (\delta_tu)^0$ and a function $u$ sufficiently smooth in $\bar{Q}_T$, and thus the approximation errors of the both equations of scheme \eqref{num3eq}-\eqref{num3ic} also have the order $O(|\*h|^4)$ as for the previous scheme \eqref{num2eq}-\eqref{num2ic}.

\par Finally, we recommend to apply scheme \eqref{num0eq}-\eqref{num0ic} only in the case $n=1$ when it takes the form \eqref{num1eq1d}-\eqref{num1ic1d}.
Instead, for $n=2$ and $3$, respectively schemes \eqref{num1eq}-\eqref{num1ic} and \eqref{num2eq}-\eqref{num2ic} can be applied.
Scheme \eqref{num3eq}-\eqref{num3ic} is more universal and can be applied for any $n\geq 1$; for $n=1$, it coincides with
\eqref{num1eq1d}-\eqref{num1ic1d} but for $n=2$ and 3 its operators are more complicated than in \eqref{num1eq}-\eqref{num1ic} and \eqref{num2eq}-\eqref{num2ic} and thus it can be more spatially dissipative in practice.

\begin{remark}
\label{rem:good-scheme}
Importantly, for example, scheme \eqref{num3eq}-\eqref{num3ic} could be derived immediately like in the second proofs of Lemmas \ref{lem:psi}-\ref{lem:psi0} by applying more
direct
though more complicated approximations of the averages in
\eqref{avereq} and \eqref{avereq_0}:
\begin{gather*}
 \bar{q}\Lambda_tu-a_i^2\bar{q}_{\hat{i}}q_t\Lambda_iu
 =\bar{s}_N\Lambda_tu-a_i^2\bar{s}_{N\hat{i}}(I+\tfrac{h_t^2}{12}\Lambda_t)\Lambda_iu+O(|\*h|^4)
\\[1mm]
 =(\bar{s}_N+\tfrac{h_t^2}{12}\bar{A}_N)\Lambda_tu+\bar{A}_Nu+O(|\*h|^4),
\\[1mm]
 \bar{q}(\delta_tu)^0-\tfrac{h_t}{2}a_i^2\bar{q}_{\hat{i}}\Lambda_iq_tu^0
 =\bar{s}_N(\delta_tu)^0-a_i^2\bar{s}_{\hat{i}}\Lambda_i\big(\tfrac{h_t}{2}u_0
 +\tfrac{h_t^2}{12}u_1
 +\tfrac{h_t^2}{12}(\delta_tu)^0\big)
\\[1mm]
 +O(|\*h|^4)=\big(\bar{s}_N+\tfrac{h_t^2}{12}\bar{A}_N\big)(\delta_tu)^0+\tfrac{h_t}{2}\bar{A}_Nu_0
 -\tfrac{h_t^2}{12}a_i^2\Lambda_iu_1+O(|\*h|^4).
\end{gather*}
\end{remark}

For $n=1$, implementation of scheme \eqref{num1eq1d}-\eqref{num1ic1d} is simple and at each time level $\{t_m\}_{m=1}^M$
comes down to solving systems of linear algebraic equations with the same tridiagonal matrix.
For $n\geq 2$, all the constructed schemes can be effectively implemented by means of solving the systems of linear algebraic equations with the same matrix arising at each time level using FFT with respect to sines in all (or $n-1$) spatial directions (after excluding the given values $\hat{v}|_{\partial\omega_h}=\hat{g}$ in the equations at the nodes closest to $\partial\omega_h$).
The matrices are non-singular (more exactly, symmetric and positive definite after the mentioned excluding) that is definitely guaranteed under the hypotheses of Theorem \ref{theo:2} below.
Note that the FFT-based algorithms have been very effective in practice in the recent study \cite{ZZ20}.
\begin{remark}
It is not difficult to extend the constructed schemes to the case of more general equation $\rho\partial_t^2u-a_i^2\partial_i^2u=f$ with $\rho=\rho(x)>0$ sufficiently smooth in $\bar{\Omega}$.
Namely, applying the alternative technique, one should simply replace the terms $s_N\Lambda_tu$, $s_N(\delta_tu)^0$ and $s_Nu_1$
with $s_N(\rho\Lambda_tu)$, $s_N(\rho(\delta_tu)^0)$ and $s_N(\rho u_1)$ in \eqref{approxerrN}, \eqref{psie0} and \eqref{u1N} keeping the same approximation orders.
Consequently the terms $s_N\Lambda_tv$, $s_N\delta_tv^0$, $\bar{s}_N\Lambda_tv$ and $\bar{s}_N\delta_tv^0$
are generalized as $s_N(\rho\Lambda_tv)$, $s_N(\rho\delta_tv^0)$, $\bar{s}_N(\rho\Lambda_tv)$ and $\bar{s}_N(\rho\delta_tv^0)$
in \eqref{num0eq}-\eqref{num0ic}, \eqref{num1eq}-\eqref{num1ic}, \eqref{num2eq}-\eqref{num2ic} and \eqref{num3eq}-\eqref{num3ic}
keeping the same approximation order $O(|\*h|^4)$.
Also the following expansions in $\Lambda_k$ for the arising operators at the upper level hold, for $n=2$ and 3, respectively
\begin{gather*}
 s_N(\rho w)+\onetwelve h_t^2A_Nw=\rho w+\onetwelve\big[h_i^2\Lambda_i(\rho w)-a_i^2h_t^2\Lambda_iw\big]
\\[1mm]
 -(\onetwelve)^2h_t^2\big(a_1^2h_2^2+a_2^2h_1^2\big)\Lambda_1\Lambda_2w,
\\[1mm]
 \bar{s}_N(\rho w)+\onetwelve h_t^2\bar{A}_Nw
 =\rho w+\onetwelve\big[h_i^2\Lambda_i(\rho w) -a_i^2h_t^2\Lambda_iw\big]
\\[1mm]
 +(\onetwelve)^2\sum_{1\leq k<l\leq 3}\big[h_k^2h_l^2\Lambda_k\Lambda_l(\rho w)-h_t^2(a_k^2h_l^2+a_lh_k^2)\Lambda_k\Lambda_lw\big]
\\[1mm]
 +(\onetwelve)^3\big[h_1^2 h_2^2h_3^2\Lambda_1\Lambda_2\Lambda_3(\rho w)
 -h_t^2(a_1^2h_2^2h_3^2+a_2^2h_1^2h_3^2+a_3^2h_1^2h_2^2)\Lambda_1\Lambda_2\Lambda_3w\big].
\end{gather*}
For $a_i$ and $h_i$ independent on $i$, the formulas are simplified, and there, on the left, the operators differ only up to factors from ones appearing in the related formulas (21)-(22) in \cite{BTT18} and (11) in \cite{STT19}.
Moreover, one can show that in this case generalized equations \eqref{num1eq} for $n=2$ and \eqref{num3eq} for $n=3$ are equivalent to respective methods from \cite{BTT18,STT19} up to approximations of $f$.
But the stability and implementation issues in the generalized case are more complicated and are beyond the scope of this paper.
\end{remark}

\par For $n\geq 2$, we also write down the scheme
\begin{gather}
 \bar{B}_N\Lambda_tv+\bar{A}_Nv=f_N\ \ \text{on}\ \ \omega_{\*h},
\label{num4eq}\\[1mm]
 v|_{\partial\omega_{\*h}}=g,\ \
 \bar{B}_N\delta_tv^0+\half h_t\bar{A}_Nv^0=u_{1N}+\half h_tf_N^0\ \ \text{on}\ \ \omega_h
\label{num4ic}
\end{gather}
with the following splitting operator at the upper time level
\begin{gather}
\hspace{-6pt} \bar{B}_N:=B_{1N}\ldots B_{nN},\
 B_{kN}:=s_{kN}-\onetwelve h_t^2a_k^2\Lambda_k=I+\onetwelve(h_k^2-h_t^2a_k^2\big)\Lambda_k.
\label{bBN}
\end{gather}
Splitting of such type is well-known and widely used, in particular, see \cite{S77,Z94}, and the implementation of this scheme is most simple and comes down to sequential solving of systems with tridiagonal matrices in all $n$ spatial directions which are definitely non-singular under the hypotheses of Theorem \ref{theo:2} below.

The following relation between $\bar{B}_N$ and $\bar{s}_N$ holds
$\bar{B}_N=\bar{s}_N+\onetwelve h_t^2\bar{A}_N+R$
with the ``residual'' operator
\begin{gather}
 R:=\sum_{k=2}^n\big(\onetwelve h_t^2\big)^k \sum_{1\leq i_1<\ldots<i_k\leq n}a_{i_1}^2\ldots a_{i_k}^2
 \Big(\prod_{1\leq l\leq n,\,l\neq i_1,\ldots i_k}s_{lN}\Big)(-\Lambda_{i_1})\ldots(-\Lambda_{i_k}).
\label{resid_oper}
\end{gather}
Clearly $R$ as the operator in $H_h$ satisfies $R=R^*>0$.
In particular,
one has
\begin{gather*}
 R=(\onetwelve h_t^2\big)^2a_1^2a_2^2\Lambda_1\Lambda_2
\ \ \text{for}\ \ n=2,
\\[1mm]
 R=(\onetwelve h_t^2\big)^2\big(a_1^2a_2^2s_{3N}\Lambda_1\Lambda_2
                               +a_1^2a_3^2s_{2N}\Lambda_1\Lambda_3
                               +a_2^2a_3^2s_{1N}\Lambda_2\Lambda_3\big)
\\[1mm]
 -(\onetwelve h_t^2\big)^3a_1^2a_2^2a_3^2\Lambda_1\Lambda_2\Lambda_3
\ \ \text{for}\ \ n=3.
\end{gather*}

Since $R\Lambda_tu=O(h_t^4)$ and $R(\delta_tu)^0=O(h_t^4)$ for a function $u$ sufficiently smooth in $\bar{Q}_T$, scheme \eqref{num4eq}-\eqref{num4ic} has the approximation error $O(|\*h|^4)$ as scheme \eqref{num3eq}-\eqref{num3ic}.
Note that some other known methods of splitting
are able to deteriorate this order of approximation.
\par Now we study the operator inequality in \eqref{ahbh} for the above arisen operators.
\begin{lemma}
\label{lem:BhAh}
For the pairs of operators $(B_h,A_h)=(s_N,A_N)$ for $n=2$, $(B_h,A_h)=(\bar{s}_N,A_N)$ for $n=3$, $(\bar{s}_N,\bar{A}_N)$ for $n\geq 1$ and $(\bar{s}_N+R,\bar{A}_N)$  for $n\geq 2$, the following inequality holds
\begin{gather}
 A_h\leq\alpha_h^2B_h\ \ \text{in}\ \ H_h\ \ \text{with}\ \ \alpha_h^2<6C_0\tfrac{a_i^2}{h_i^2},
\label{BhAh}
\end{gather}
where $C_0=\tfrac43$ in the first case of $(B_h,A_h)$
or $C_0=1$ in other cases.
\end{lemma}
\begin{proof}
Let $1\leq k\leq n$ and $\{\lambda_l^{(k)}:=\tfrac{4}{h_k^2}\sin^2\tfrac{\pi lh_k}{2X_k}\}_{l=1}^{N_k-1}$
be the collection of eigenvalues of the operator $-\Lambda_k$ in $H_h$, with the maximal of them
$\lambda_{\max}^{(k)}=\tfrac{4}{h_k^2}\sin^2\tfrac{\pi(N_k-1)}{2N_k}<\tfrac{4}{h_k^2}$.
The inequality $-\Lambda_k\leq\alpha_{1h}^2s_{kN}$ in $H_h$ is equivalent to the following inequality between the eigenvalues of these operators
\[
 \lambda_l^{(k)}\leq\alpha_{1h}^2\big(1-\onetwelve h_k^2\lambda_l^{(k)}\big),\ \ 1\leq l\leq N_k-1.
\]
Consequently the sharp constant is
\[
 \alpha_{1h}^2=\max_{1\leq l\leq N_k-1}\frac{\lambda_l^{(k)}}{1-\onetwelve h_k^2\lambda_l^{(k)}}
 =\frac{\lambda_{\max}^{(l)}}{1-\onetwelve h_k^2\lambda_{\max}^{(k)}}
 <\tfrac32\lambda_{\max}^{(k)}
 <6\tfrac{1}{h_k^2}.
\]
Herewith $\alpha_{1h}^2=6\tfrac{1}{h_k^2}\big(1+O\big(\tfrac{1}{N_k^2}\big)\big)$, thus the last bound is asymptotically sharp.

Similarly for $n=2$ the inequality $A_N\leq\alpha_h^2s_N$ in $H_h$ holds with
\begin{gather*}
 \alpha_h^2=\max_{1\leq k\leq N_1-1,\,1\leq l\leq N_2-1}
 \frac{\big(1-\onetwelve h_1^2\lambda_k^{(1)}\big)a_2^2\lambda_l^{(2)}
      +\big(1-\onetwelve h_2^2\lambda_l^{(2)}\big)a_1^2\lambda_k^{(1)}}
      {1-\onetwelve h_1^2\lambda_k^{(1)}-\onetwelve h_2^2\lambda_l^{(2)}}.
\end{gather*}
It is not difficult to check that the function under the $\max$ sign has the positive partial derivatives with respect to arguments
$\lambda_k^{(1)}$ and $\lambda_l^{(2)}$ on the natural intervals of their values and thus
\begin{gather*}
\alpha_h^2=\frac{\big(1-\onetwelve h_1^2\lambda_{\max}^{(1)}\big)a_2^2\lambda_{\max}^{(2)}
      +\big(1-\onetwelve h_2^2\lambda_{\max}^{(2)}\big)a_1^2\lambda_{\max}^{(1)}}
      {1-\onetwelve h_1^2\lambda_{\max}^{(1)}-\onetwelve h_2^2\lambda_{\max}^{(2)}}
      <2\big(a_1^2\lambda_{\max}^{(1)}+a_2^2\lambda_{\max}^{(2)}\big).
\end{gather*}
This implies
\eqref{BhAh} in the first case. The last bound is asymptotically sharp~ too.

\par Next, in virtue of the inequalities $s_{N\hat{i}}<\bar{s}_{N\hat{i}}$ for $n=3$ (see formula \eqref{bsNsN} for $n=2$) and
$-\Lambda_k<\tfrac32\lambda_{\max}^{(k)}s_{kN}$ in $H_h$, the following inequalities in $H_h$ hold:
\begin{gather*}
 A_N=-a_i^2s_{N\hat{i}}\Lambda_i
 <a_i^2\bar{s}_{N\hat{i}}\big(\tfrac32 \lambda_{\max}^{(i)}s_{iN}\big)
 =\tfrac32\big(a_i^2\lambda_{\max}^{(i)}\big)\bar{s}_N\ \ \text{for}\ \ n=3,
\\[1mm]
 \bar{A}_N=-a_i^2\bar{s}_{N\hat{i}}\Lambda_i
 <a_i^2\bar{s}_{N\hat{i}}\big(\tfrac32\lambda_{\max}^{(i)}s_{iN}\big)
 =\tfrac32\big(a_i^2\lambda_{\max}^{(i)}\big)\bar{s}_N
 \leq\tfrac32\big(a_i^2\lambda_{\max}^{(i)}\big)(\bar{s}_N+R)
\end{gather*}
for $n\geq 2$.
Therefore inequality \eqref{BhAh} has been proved in all the cases.
\end{proof}

Now we state a result on conditional stability in two norms for the constructed schemes.
\begin{theorem}
\label{theo:2}
Let $g=0$ in \eqref{hyperb2ibc} and $0<\ve_0<1$.
Let us consider schemes \eqref{num1eq}-\eqref{num1ic}, \eqref{num2eq}-\eqref{num2ic},
\eqref{num3eq}-\eqref{num3ic} and \eqref{num4eq}-\eqref{num4ic} under the condition
\begin{gather}
 C_0h_t^2\tfrac{a_i^2}{h_i^2}\leq 1-\ve_0^2
\label{condhth2}
\end{gather}
with the pairs of operators respectively
$(B_h,A_h)=(s_N,A_N)$ for $n=2$,
$(B_h,A_h)=(\bar{s}_N,A_N)$ for $n=3$,
$(\bar{s}_N,\bar{A}_N)$ for $n\geq 1$ (for $n=1$, this covers also scheme \eqref{num1eq1d}-\eqref{num1ic1d}) and $(\bar{s}_N+R,\bar{A}_N)$ for $n\geq 2$.
Here $C_0$ is the same as in Lemma~ \ref{lem:BhAh}.

\par Then the solutions to all the listed schemes satisfy the following bounds
\begin{gather}
\max_{1\leq m\leq M}
\big[\ve_0^2\|\bar{\delta}_tv^m\|_{B_h}^2
+\|\bar{s}_tv^m\|_{A_h}^2\big]^{1/2}
\nonumber\\[1mm]
\leq\big(\|v^0\|_{A_h}^2+\ve_0^{-2}\|B_h^{-1/2}u_{1N}\|_h^2\big)^{1/2}
+2\ve_0^{-1}\|B_h^{-1/2}f_N\|_{L_{h_t}^1(H_h)};
\label{energy est1N}
\end{gather}
the $f_N$-term
can be taken as
$2I_{h_t}^{M-1}\|A_h^{-1/2}\bar{\delta}_tf_N\|_h
+3\max\limits_{0\leq m\leq M-1}\|A_h^{-1/2}f^m\|_h$ as well, and
\begin{gather*}
\max_{0\leq m\leq M}\max\big\{\ve_0\|v^m\|_{B_h},\,\|I_{h_t}^m\bar{s}_tv\|_{A_h}\big\}
\\[1mm]
 \leq\|v^0\|_{B_h}
 +2\|A_h^{-1/2}u_{1N}\|_h
 +2\|A_h^{-1/2}f_N\|_{L_{h_t}^1(H_h)};
\label{energy est2N}
\end{gather*}
for $f_N=\delta_tg$, the $f_N$-term
can be replaced with $2\ve_0^{-1}I_{h_t}^{M}\|B_{h}^{-1/2}\big(g-s_tg^0\big)\|_h$.

Importantly, the both bounds hold for any free terms $u_{1N}\in H_h$ and $f_N$: $\{t_m\}_{m=0}^{M-1}\to H_h$ (not only for those defined in Lemmas \ref{lem:psi}-\ref{lem:psi0}).
\end{theorem}
\begin{proof}
The theorem follows immediately from the above general stability Theorem \ref{theo:1} applying assumption \eqref{stabcond} for $\sigma=1/12$, in virtue of inequality \eqref{ve_0ineq} and Lemma \ref{lem:BhAh}.
\end{proof}
\begin{corollary}
For the sufficiently smooth in $\bar{Q}_T$ solution $u$ to the IBVP
\eqref{hyperb2eq}-\eqref{hyperb2ibc}, $v^0=u_0$ on $\omega_h$ and under the hypotheses of Theorem \ref{theo:2} excluding $g=0$, for all the schemes listed in it, the following 4th order error bound in the strong energy norm holds
\[
 \max_{1\leq m\leq M}\big[\ve_0^2\|\bar{\delta}_t(u-v)^m\|_{B_h}^2+\|\bar{s}_t(u-v)^m\|_{A_h}^2\big]^{1/2}=O(|h|^4).
\]
\end{corollary}

The proof is standard (for example, see \cite{S77}) and follows from the stability bound \eqref{energy est1N} applied to the error $r:=u-v$ (herewith $r|_{\partial\omega_{\*h}}=0$, $r^0=0$).
The approximation errors play the role of $f_N^m$, $1\leq m\leq M-1$, and $u_{1N}$ in the equations of the schemes, and the above checked conclusion that they have the order $O(|\*h|^4)$ for all the listed schemes is essential, as well as $h_t=O(|h|)$ in Theorem \ref{theo:2}.

Notice that,
in the very particular case
$\tfrac{h_1}{a_1}=\ldots\tfrac{h_n}{a_n}=h_t$,
schemes \eqref{num1eq1d}-\eqref{num1ic1d} and \eqref{num4eq}-\eqref{num4ic} become \textit{explicit} (since then $\bar{B}_N=I$, see \eqref{bBN}) and, moreover, the latter one differs from the simplest explicit scheme only by the above derived approximations of the free terms in its equations.
Herewith, for scheme \eqref{num1eq1d}-\eqref{num1ic1d}, condition \eqref{condhth2} is valid with $C_0=1$ and only $\ve_0=0$ (actually, with some $0<\ve_0=\ve_0(h)<1$ as one can check).
But, for scheme \eqref{num4eq}-\eqref{num4ic} and $n\geq 2$, the condition even with $\ve_0=0$ fails; more careful analysis of inequality \eqref{BhAh} for this scheme still allows to improve the bound for $\alpha_h^2$ but not the drawn conclusion itself.
According to Remark \ref{rem1}, for scheme \eqref{num1eq1d}-\eqref{num1ic1d}, even in this particular case some stability bounds still hold.
The bounds contain terms of the following type
\begin{gather*}
 \|w\|_{0,\*h}^2=\big((I+\tfrac{h_i^2}{4}\Lambda_1)w,w\big)_h
 \geq\cos^2\tfrac{\pi(N_1-1)}{2N_1}\|w\|_h^2\ \ \forall w\in H_h.
\end{gather*}
Thus $\|w\|_{0,\*h}$ remains a norm in $H_h$ but clearly is no longer bounded from below by $\|w\|_h$ uniformly in $h$ (since the constant in the last inequality is sharp and has the order $O\big(\frac{1}{N_1^2}\big)$).

\par The explicit scheme for $n=1$ is very specific.
Its equations are rewritten using a 4-point stencil
simply as
\begin{gather}
 v_k^{m+1}=v_{k-1}^m+v_{k+1}^m-v_k^{m-1}+h_t^2f_{Nk}^m\ \ \text{on}\ \ \omega_{\*h},\ \ 1\leq m\leq M-1,
\label{explicit1}\\[1mm]
 v|_{\partial\omega_{\*h}}=g,\ \ v_k^1=\half(v_{k-1}^0+v_{k+1}^0)+h_tu_{1Nk}+\half h_t^2f_{Nk}^0\ \ \text{on}\ \ \omega_h.
\label{explicit2}
\end{gather}
For clarity, let us pass to the related Cauchy problem with any $k\in\mathbb{Z}$, $x_k=kh$, $h=h_1$, $a=a_1$ and the omitted boundary condition.
Then the following explicit formula holds
\begin{gather*}
 v_k^m=\half(v_{k-m}^0+v_{k+m}^0)
 +\sum\nolimits_{l\in I_k^{m}}h_tu_{1Nl}+\half h_t^2f_{Nl}^0
 +h_t^2\sum_{p=1}^{m-1}\sum\nolimits_{l\in I_k^{m-p}}f_{Nl}^p,
\end{gather*}
where $k\in\mathbb{Z}$, $1\leq m\leq M$ and $I_k^{m-p}$ is the set of indices
$\{k-(m-p-1),k-(m-p+1),\ldots,k+(m-p-1)\}$.
It can be verified most simply by induction with respect to $m$.
Notice that all the mesh nodes lie on the characteristics $x-x_k=\pm a_1t$ of the equation.
Of course, the stability of the scheme can be directly proved applying this formula.
\par Let us take $v^0=u_0$ and reset $u_{1Nk}=\frac{1}{2h}\int_{x_{k-1}}^{x_{k+1}}u_1(x)\,dx$ and
\[
 f_{Nk}^0=\frac{1}{hh_t}\int_{T_k^1}f(x,t)\,dxdt,\ \ f_{Nk}^m=\frac{1}{2hh_t}\int_{R_k^m}f(x,t)\,dxdt;
\]
here $T_k^m$ and $R_k^m$ are the triangle and rhomb with the vertices $\{(x_{k\pm m},0)$, $(x_k,t_m)\}$ and
$\{(x_{k\pm 1},t_m),(x_k,t_{m\pm 1})\}$.
Then the above formula for $v_k^m$ takes the form
\[
 v_k^m=\half\big(u_0(x_{k-m})+u_0(x_{k+m})\big)
 +\tfrac{1}{2a_1}\int_{x_{k-m}}^{x_{k+m}}u_1(x)\,dx
 +\tfrac{1}{2a_1}\int_{T_k^m}f(x,t)\,dxdt,
\]
thus at the mesh nodes it reproduces the classical d'Alembert formula for the solution $u$ to the Cauchy problem for the 1D wave equation, where the approximate and exact solution \textit{coincide}: $v_k^m\equiv u(x_k,t_m)$ for any $k\in\mathbb{Z}$ and $0\leq m\leq M$.
Concerning exact schemes, see also \cite{LMP2016}.

\section{The case of non-uniform rectangular meshes}
\label{nonunif_mesh}
\setcounter{equation}{0}
\setcounter{lemma}{0}
\setcounter{theorem}{0}

This section is devoted to a generalization to the case of non-uniform rectangular meshes.
Let $1\leq k\leq n$.
Define the general non-uniform meshes $0=x_{k0}<x_{k1}<\ldots<x_{kN_k}=X_k$ in $x_k$ with the steps $h_{kl}=x_{kl}-x_{k(l-1)}$ and
$\overline\omega_{h_t}$ with the nodes $0=t_0<t_1<\ldots<t_M=T$ and steps $h_{tm}=t_m-t_{m-1}$.
Let $\omega_{hk}=\{x_{kl}\}_{l=1}^{N_k-1}$.
We set
\[
 h_{k+,l}=h_{k(l+1)},\ \ h_{*k}=\half(h_k+h_{k+}),\ \ h_{t+,m}=h_{t(m+1)},\ \ h_{*t}=\half(h_t+h_{t+})
\]
as well as $h_{k\max}=\max_{1\leq l\leq N_k} h_{kl}$ and $h_{t\max}=\max_{1\leq m\leq M} h_{tm}$.
Define the difference operators
\begin{gather*}
 \delta_kw_l=\tfrac{1}{h_{k+,l}}(w_{l+1}-w_l),\ \
 \bar{\delta}_kw_l=\tfrac{1}{h_{kl}}(w_l-w_{l-1}),\ \ \Lambda_kw=\tfrac{1}{h_{*k}}(\delta_kw-\bar{\delta}_kw),
\\[1mm]
 \delta_ty^m=\tfrac{1}{h_{t+,m}}(y^{m+1}-y^m),\ \
 \bar{\delta}_ty^m=\tfrac{1}{h_{tm}}(y^m-y^{m-1}),\ \ \Lambda_ty=\tfrac{1}{h_{*t}}(\delta_ty-\bar{\delta}_ty),
\end{gather*}
where $w_l=w(x_{kl})$ and $y^m=y(t_m)$.
The last four operators generalize those defined above so that their notation is the same.

\par We extend the above technique based on averaging equation \eqref{hyperb2eq} and generalize the above average in $x_k$:
\begin{gather*}
 q_kw(x_{kl})=\frac{1}{h_{*k,l}}\int_{x_{k(l-1)}}^{x_{k(l+1)}}w(x_k)e_{kl}(x_k)\,dx_k,
\\[1mm]
e_{kl}(x_k)=\tfrac{x_k-x_{k(l-1)}}{h_{kl}}\,\ \text{on}\ [x_{k(l-1)},x_{kl}],\
e_{kl}(x_k)=\tfrac{x_{k(l+1)}-x_k}{h_{k+,l}}\,\ \text{on}\ [x_{kl},x_{k(l+1)}].
\end{gather*}

For a function $w(x_k)$ smooth on $[0,X_k]$, formula \eqref{qklambdak} remains valid and
\begin{gather*}
 q_kw=w+q_k\rho_{k1}(\partial_kw),
\nonumber\\[1mm]
 q_kw=w+\tfrac13(h_{k+}-h_k)\partial_kw+\onetwelve\big[(h_{k+})^2-h_{k+}h_k+h_k^2\big]\partial_k^2w+q_k\rho_{k3}(\partial_k^3w)
\label{qkexpansion}
\end{gather*}
on $\omega_{hk}$,
and the first bound \eqref{resid_bound} remains valid for $s=1,3$ with $h_k$ replaced with $h_{*k}$, see also \eqref{taylor_residual},
that follows from Taylor's formula after calculating the arising integrals over $[x_{k(l-1)},x_{k(l+1)}]$.
Due to Taylor's formula we also have
\begin{gather*}
 \partial_kw=\half(\bar{\delta}_kw+\delta_kw)-\tfrac14(h_{k+}-h_k)\partial_k^2w+\rho_{k}^{(1)}(\partial_k^3w),
 \partial_k^2w=\Lambda_kw+\rho_{k3}^{(2)}(\partial_k^3w),
\\[1mm]
 |\rho_{k}^{(s)}(\partial_k^3w)|\leq c^{(s)}h_{*k}^{3-2(s-1)}\|\partial_k^3w\|_{C(I_{kl})},\ \ s=1,2,
\end{gather*}
thus the second expansion for $q_kw$
implies that
\begin{gather*}
 q_kw=s_{kN}w+\tilde{\rho}_{k3}(\partial_k^3w),\ \
 |\tilde{\rho}_{k3}(\partial_k^3w)|\leq \tilde{c}_3h_{*k}^3\|\partial_k^3w\|_{C(I_{kl})},
\\[1mm]
 s_{kN}:=I+\tfrac13(h_{k+}-h_k)\big[\half(\bar{\delta}_k+\delta_k)-\tfrac14(h_{k+}-h_k)\Lambda_k\big]
\nonumber\\[1mm]
 +\onetwelve\big[(h_{k+})^2-h_{k+}h_k+h_k^2\big]\Lambda_k
 =I+\tfrac16(h_{k+}-h_k)(\bar{\delta}_k+\delta_k)+\onetwelve h_kh_{k+}\Lambda_k,
\end{gather*}
i.e., $s_{kN}=I+\onetwelve(h_{k+}\beta_k\delta_k-h_k\alpha_k\bar{\delta}_k)$
or, in the averaging form,
\begin{gather*}
 s_{kN}w_l=\onetwelve(\alpha_{kl}w_{l-1}+10\gamma_{kl}w_l+\beta_{kl}w_{l+1}),
\\
 \alpha_k=2-\tfrac{h_{k+}^2}{h_kh_{*k}},\ \beta_k=2-\tfrac{h_k^2}{h_{k+}h_{*k}},\ \gamma_k=1+\tfrac{(h_{k+}-h_k)^2}{5h_k h_{k+}},\ \alpha_k+10\gamma_k+\beta_k=12;
\end{gather*}
all the presented formulas are valid on $\omega_{hk}$.
The operator $s_{kN}$ generalizes one defined above.
Its another derivation was originally given in \cite{JIS84}, see also \cite{ChS18,RCM14}.
Recall that the natural property $\alpha_{kl}\geq 0$ and $\beta_{kl}\geq 0$ (not imposed below) is equivalent to the rather restrictive condition
on the ratio of the adjacent mesh steps
\begin{equation}
0.618\approx \tfrac{2}{\sqrt{5}+1}\leq\tfrac{h_{k(l+1)}}{h_{kl}}\leq\tfrac{\sqrt{5}+1}{2}\approx 1.618.
\label{step ratio cond}
\end{equation}
\par On $\omega_{h_t}$, the average $q_tw=q_{n+1}w$ is defined similarly, and thus
\begin{gather*}
 q_tw=s_{tN}w+\tilde{\rho}_{t3}(\partial_t^3w),\ \
|\tilde{\rho}_{t3}(\partial_t^3w)|\leq \tilde{c}_3h_{*t}^3\|\partial_t^3w\|_{C[t_{m-1},t_{m+1}]}
\end{gather*}
with $s_{tN}=I+\onetwelve(h_{t+}\beta_t\delta_t-h_t\alpha_t\bar{\delta}_t)$ or, in the averaging form,
\[
 s_{tN}y=\onetwelve(\alpha_t\check{y}+10\gamma_ty+\beta_t\hat{y}),\,
 \alpha_t=2-\tfrac{h_{t+}^2}{h_th_{*t}},\ \beta_t=2-\tfrac{h_t^2}{h_{t+}h_{*t}},\,
 \gamma_t=1+\tfrac{(h_{t+}-h_t)^2}{5h_t h_{t+}}.
\]

\par Let $\omega_{h}=\omega_{h1}\times\ldots\times\omega_{hn}$.
Formula \eqref{avereq} for $u$ remains valid
and implies
\[
 \bar{s}_N\Lambda_tu-a_i^2\bar{s}_{N\hat{i}}s_{tN}\Lambda_iu=\bar{q}q_tf+O(\*h_{\max}^3)\ \ \text{on}\ \ \omega_{\*h},
\]
where $\*h_{\max}=\max\{h_{1\max},\ldots,h_{n\max},h_{t\max}\}$.
Formula \eqref{avereq_0} for $u$ remains valid as well.
It involves only two first time levels thus easily covers the case of the non-uniform mesh in $t$ and implies now
\[
 \big(\bar{s}_N+\tfrac{h_{t1}^2}{12}a_i^2\bar{s}_{N\hat{i}}\Lambda_i\big)(\delta_tu)^0
 +\tfrac{h_{t1}}{2}a_i^2\bar{s}_{N\hat{i}}\Lambda_iu_0
 =\bar{q}u_1+\tfrac{h_{t1}^2}{12}a_i^2\bar{s}_{N\hat{i}}\Lambda_iu_1+\bar{q}q_tf^0+O(\*h_{\max}^3)
\]
on $\omega_h$,
where $q_ty^0$ is given by
formula \eqref{qty0}
with $h_{t1}$ in the role of $h_t$.

\par Owing to the above formulas, see also Remark \ref{rem:good-scheme}, the last two formulas with $u$ lead us to the generalized scheme \eqref{num3eq}-\eqref{num3ic}:
\begin{gather}
 \tfrac{1}{h_{*t}}\big[\big(\bar{s}_N+\tfrac{h_{*t}h_{t+}}{12}\beta_t\bar{A}_N\big)\delta_tv
 -\big(\bar{s}_N+\tfrac{h_{*t}h_t}{12}\alpha_t\bar{A}_N\big)\bar{\delta}_tv\big]
 +\bar{A}_Nv=\bar{s}_Ns_{tN}f,
\label{num3eq nonuni}\\
\hspace{-8pt} v|_{\partial\omega_{\*h}}=g,\ \big(\bar{s}_N+\tfrac{h_{t1}^2}{12}\bar{A}_N\big)(\delta_tv)^0
 +\tfrac{h_{t1}}{2}\bar{A}_Nv_0
 =(\bar{s}_N-\tfrac{h_{t1}^2}{12}\bar{A}_N)u_1+\tfrac{h_{t1}}{2}f_N^0
\label{num3ic nonuni}
\end{gather}
with $f_N^0=\bar{s}_Nf_0+\tfrac{h_{t1}}{3}(\delta_tf)^0$,
where equations are valid respectively on $\omega_{\*h}$ and $\omega_h$ and have the approximation errors of the order $O(\*h_{\max}^3)$.
\par For the uniform mesh in $t$, the left-hand side of \eqref{num3eq nonuni} takes the previous form whereas the term $\bar{s}_Ns_{tN}f$ can be simplified keeping the same order of the approximation error:
\begin{gather}
 (\bar{s}_N+\onetwelve h_t^2\bar{A}_N)\Lambda_tv+\bar{A}_Nv=(\bar{s}_N+\onetwelve h_t^2\Lambda_t)f.
\label{num3eq nonuni in x}
\end{gather}

\par The splitting version of equation \eqref{num3eq nonuni} can be got by replacing the operators in front of $\delta_tv$ and $\bar{\delta}_tv$ by the operators of the form
\[
 \bar{B}_N=(s_1+\onetwelve h_{*t}\tilde{h}_t\sigma_ta_1^2\Lambda_1)\ldots(s_n+\onetwelve h_{*t}\tilde{h}_t\sigma_ta_n^2\Lambda_n),
\]
where respectively $\tilde{h}_t=h_{t+}$ and $\sigma_t=\beta_t$, or $\tilde{h}_t=h_t$ and $\sigma_t=\alpha_t$.
Since
\[
 \bar{B}_N=\bar{s}_N+\onetwelve h_{*t}\tilde{h}_t\sigma_ta_i^2\bar{s}_{N\hat{i}}\Lambda_i+R,
\]
where the operator $R$ satisfies formula \eqref{resid_oper} with $h_t^2$ replaced with $h_{*t}\tilde{h}_t\sigma_t$,
this replacement conserves the approximation error of the order $O(\*h_{\max}^3)$.
The splitting version of equation \eqref{num3ic nonuni} is got simply by replacing $\bar{s}_N+\tfrac{h_{t1}^2}{12}\bar{A}_N$
with the above operator \eqref{bBN} with $h_{t1}$ in the role of $h_t$.

\par One can check also that the approximation errors still has the 4th order $O(\*h_{\max}^4)$ for smoothly varying non-uniform meshes, cp. \cite{Z15}, provided that, for example, $f_N^0=\bar{s}_Nf^0-f^0+f_{dh_t}^{(0)}$.

\par Here we do not touch the stability study in the case of the non-uniform mesh (even only in space) but this is noticeably more cumbersome like in \cite{Z15} (since the operator $s_{kN}$ is not self-adjoint any more) and, moreover,
imposes stronger conditions on $h_t$,
see also \cite{ZC18,ZC20}.

\section{Numerical experiments}
\label{numerexperiments}
\setcounter{equation}{0}
\setcounter{lemma}{0}
\setcounter{theorem}{0}

\textbf{\ref{numerexperiments}.1.} In the  IBVP \eqref{hyperb2eq}-\eqref{hyperb2ibc} in the 1D case, we now take $\Omega:=(-X/2,X/2)$ and rewrite the boundary condition as $u|_{x=-X/2}=g_0(t)$ and $u|_{x=X/2}=g_1(t)$, $t\in(0,T)$.
We intend to analyze the practical error orders $\gamma_{pr}$ of $r=u-v$ in three uniform in time mesh norms
\begin{equation}
 \max_{0\leq m\leq M}\|r^m\|_h,\ \
 \max_{0\leq m\leq M,\,0\leq k\leq N}|r_k^m|,\ \
\max_{1\leq m\leq M}\max\big\{\|\bar{\delta}_tr^m\|_h,\,\|\bar{\delta}_1r^m\|_{\tilde{h}}\big\},
\label{3norms}
\end{equation}
which below are denoted respectively as $L_h^2$, $C_h$ and $\mathcal{E}_h$ (the 2nd and 3rd norms are the uniform and strong energy-type ones).
Here $\|w\|_{\tilde{h}}=\big(h\sum_{k=1}^{N}w_k^2\big)^{1/2}$ and $N=N_1$.
The respective expected theoretical error orders $\gamma_{th}$ are
\begin{gather}
 \min\big\{\tfrac45\alpha,4\big\},\,\ \alpha\geq 0;\ \
 \tfrac45(\alpha-\tfrac12),\,\
\tfrac12<\alpha\leq\tfrac{11}{2};\ \
 \tfrac45(\alpha-1),\,\ 1\leq\alpha\leq 6
\label{gammath}
\end{gather}
(in the spirit of \cite{BTW75}),
where $\alpha$ is the parameter defining the weak smoothness of the data, see details below
(concerning the first order, for $\alpha\leq 1$, it should refer to the continuous $L^2$ norm rather than the mesh one but that we will ignore).
The proof of the first order in the case $u_1=f=0$ see in \cite{J94}.
For comparison, recall that for the 2nd approximation order methods the corresponding
theoretical
error orders $\gamma_{th}^{(2)}$ are
\begin{equation}
 \min\big\{\tfrac23\alpha,2\big\}, \alpha\geq 0;\
 \min\big\{\tfrac23(\alpha-\tfrac12),2\big\}, \alpha>\tfrac12;\
 \min\big\{\tfrac23(\alpha-1),2\big\}, \alpha\geq 1,
\label{gammath2}
\end{equation}
according to \cite{Z94}; recall that the middle error order is derived from two other ones.
These orders also have recently been confirmed practically
in \cite{ZKMMA2018}.

\par Let $P_0(x)=(\sgn x+1)/2$ be the Heaviside-type function, $P_1(x)=1-2|x|$,
$P_k(x)=(\sgn x)(2x)^k$ ($k\geq 2$) and $Q_l(t)=                                                                           P_0(t-t_*)(t-t_*)^l$ ($l\geq 0$ and $0<t_*<T$) be piecewise-polynomial functions.
For uniformity, we also set $P_{-1}(x)=\delta(x)$ and $Q_{-1}(t)=\delta(t-t_*)$ as the Dirac delta-functions  concentrated at $x=0$ and $t=t_*$.
We put $X=T=1$.

\par We consider six typical Examples $E_{\alpha}$, $\alpha=1/2,3/2,\ldots,11/2$, of non-smooth data supplementing the study in \cite{ZKMMA2018}.
The initial functions $u_0=P_{[\alpha]}$ and $u_1=c_1P_{[\alpha]-1}$
are piecewise-polynomial functions of the degree $[\alpha]$ and $[\alpha]-1$ respectively, with a unique singularity point $x=0$, excluding the case $[\alpha]=0$ for $u_1$, where $u_1(x)=c_1\delta(x)$.
Thus $u_0$ belongs to the Nikolskii space $H_2^\alpha(\Omega)$ \cite{N75} (and to the Sobolev-Slobodetskii space $W_2^{\alpha-\ve}(\Omega)$, $0<\ve<1/2$), and $u_1\in H_2^{\alpha-1}(\Omega)$ (for $\alpha>1$).

\par The free term $f(x,t)=c_2P_{-1}(x)Q_{-1}(t)=c_2\delta(x,t-t_*)$
is concentrated at $(x,t)=(0,t_*)$ for $\alpha=1/2$,
or has the form $f(x,t)=f_1(x)f_2(t)=c_2P_0(x)Q_{-1}(t)$ for $\alpha=3/2$,
or the form of two such type summands $f(x,t)=c_2P_0(x)Q_{[\alpha]-2}(t)+c_3P_1(x)Q_{[\alpha]-3}(t)$
for $\alpha\geq 5/2$.
The term $f_1$ is piecewise-constant (the case $\alpha_1=1/2$) for $\alpha\geq 3/2$ or also piecewise-linear (the case $\alpha_1=3/2$) for $\alpha\geq 5/2$, with a unique singularity point $x=0$.
Respectively the term $f_2(t)=\delta(t-t_*)$
for $\alpha=3/2$ or $f_2(t)=Q_{\alpha_2}(t)$ is a piecewise-polynomial function of the degree $\alpha_2=\alpha-2-\alpha_1=[\alpha]-2-[\alpha_1]$ for $\alpha\geq 5/2$, with a unique singularity point $t=t_*$.
Recall that, for $\alpha_2>0$, such $f_2$ belongs to the Sobolev-Nikolskii space $WH_1^{\alpha_2}(0,T)$
(for example, see \cite{Z94}), though not to the less broad Sobolev space $W_1^{\alpha_2}(0,T)$.
Thus $f$ itself or its both summands has the so called weak \textit{dominated mixed smoothness} of the order $\alpha_1$ in $x$ and $\alpha_2$ in $t$, with $\alpha_1+\alpha_2=\alpha-1$.
Recall that this property is much broader than the standard weak smoothness of the order $\alpha-1$ in both $x$ and $t$ in $L^2(Q)$; in particular, the case of $f$ discontinuous in $x$ is covered for \textit{any} considered $\alpha$.
\par Here $(c_1,c_2)=(0.4,0.4),(1.9,1.1)$
and
$(c_1,c_2,c_3)=
(0.58,2.1,2.3),
(2.8,6.8$, $7.3)$,
$(3.7,13,31)$,
$(4.6,24,51)$
for $\alpha=\frac12,\frac32,\ldots,\frac{11}{2}$ respectively.
We use these multipliers to make the contributions to $r(\cdot,T)$ due to $u_0$, $u_1$ and $f$ of the similar magnitude and thus
significant.
\par We also take smooth $g_0$ and $g_1$ (not affecting $\gamma_{pr}$) to simplify the explicit forms of $u$ (which we omit here) based on the d'Alembert formula.
Namely, we set
$g_0=0$ and $g_1(t)=(c_1t)^k$ for $\alpha=\frac12,\frac32$;
$g_0=(-1)^{k}(-g_1^{(0k)}+c_1g_1^{(1k)})$ and $g_1=g_1^{(0k)}+c_1g_1^{(1k)}$ for $\alpha\geq\frac52$, where $k=[\alpha]$ and
\begin{gather*}
 g_1^{(0k)}(t)=\half\big[(1-2at)^k+(1+2at)^k\big],\ \ k\geq 2,
\\
 g_1^{(1k)}(t)=0,\ \ k=2,\ \
 g_1^{(1k)}(t)=\tfrac{1}{4ak}\big[(1+2at)^k-(1-2at)^k\big],\ \ k\geq 3.
\end{gather*}
\par The properties of $u$ in Example $E_{\alpha}$ have been described in \cite{ZKMMA2018} or are similar.
Recall that, for example, $u$ is piecewise-constant and discontinuous on $\bar{Q}$ for $\alpha=\frac12$, or $u$ is piecewise-linear with  discontinuous piecewise-constant derivatives on $\bar{Q}$ for $\alpha=\frac32$, etc.
The straight singularity lines are characteristics and $t=t_*$.
Notice that $u$ is not the classical solution \textit{for any} $\alpha$ but is strong one for $\alpha\geq\frac72$
and one of several weak solutions for $\alpha\leq\frac52$, see details in \cite{ZKMMA2018}
(but note that, for $f(x,t)=P_1(x)\delta(t-t_*)$, $\alpha=\frac52$, the jump of $\partial_tu$ across $t=t_*$ was not taken into account there).

\par We set $v^0=u_0$;
$u_{1N}=q_1u_1$ for $\alpha\leq\frac32$, or as in \eqref{u1N} for $\alpha\geq\frac52$,
and $f_N^m=(q_1f_1)q_tf_2^m$ on $\omega_{\*h}$, see Remark \ref{rem:nonsmooth_f}.
For $x_k\in\omega_h$ and even $N$, we have
$q_1\delta(\cdot)_k=\tfrac{1}{h}$ for $x_k=0$, or $q_1\delta(\cdot)_k=0$ otherwise, and $(q_1P_0)_k=P_0(x_k)$.
Also, if $t_*=t_{m_*}\in\omega_{h_t}$ and $1\leq k\leq 5$, then
$(q_tQ_k)^m=(s_{tN}Q_k)^m$ for $t_m\in\omega_{h_t}$,
$m\neq m_*$, or
$(q_tQ_k)^{m_*}=\tfrac{\tau^k}{(k+1)(k+2)}$.
We choose $a=\frac{1}{\sqrt5}$, $t_*=\frac{T}{2}$ and $\tau=h$ (so the mesh is not adjusted to the characteristics).

\par To identify error orders more reliably, we compute the errors for respectively $N=200,400$, $\ldots,N_{\max}$,
where $N_{\max}=3200,2000,800$ respectively for $\frac12\leq\alpha\leq\frac52,\alpha=\frac72,\frac92$; also $N=200,300,\ldots,N_{\max}$ with $N_{\max}=600$ for $\alpha=\frac{11}{2}$
($N_{\max}$ is lesser for $\alpha\geq\frac72$ to avoid an impact of the round-off errors on ${\gamma_{pr}}$).
We plot graphs of $\log_{10}\|r\|$ versus $\log_{10}N$,
where $\|r\|$ is each of the three norms \eqref{3norms}, and seek the almost linear dependence between them by the least square method.
Thus we calculate the dependence $\|r\|\approx c_0h^{\gamma_{pr}}=c_0(\frac{X}{N})^{\gamma_{pr}}$.

For $\alpha=\frac32,\frac52,\frac72,\frac{9}{2}$ and the extended set $N=200,400,\ldots,3200$,
we present
$\mathcal{E}$,
$C_h$, $L_h^2$-norms of the error denoted respectively by $\triangle, \Box, \diamondsuit$ on Figs. \ref{Fig1}-\ref{Fig2}.
Notice the abrupt decrease of the error range as $\alpha$ grows.
We also observe the slight oscillation of the data for $\alpha=\frac32$ that is an exception
(they also present for $\alpha=\frac12$);
instead, the linear behavior is typical for other $\alpha$ and the values $200\leq N\leq N_{\max}$.
The slight growth of $L_h^2$-norm for $\alpha=\frac72$ and much more significant growth of all the norms for $\alpha=\frac{9}{2}$ as $N$ increases reflect the impact of the round-off errors; the value of $N$ when the error begins to increase depends on the norm.
For $\alpha=\frac{11}{2}$ the situation is even more strong (not presented).
\begin{figure}[h]
\begin{minipage}[t]{0.48\linewidth}
\centering
\includegraphics[width=1\linewidth]
{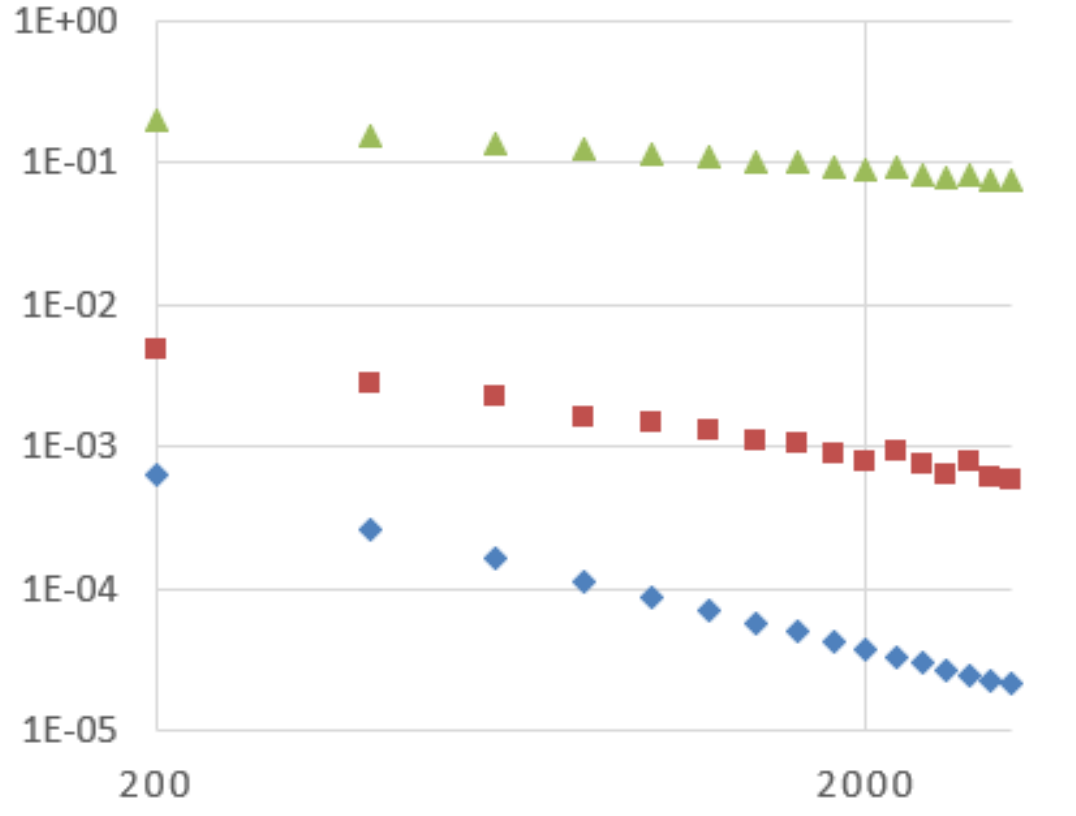}
 $a)\ \alpha=3/2$
\end{minipage}
\hfill
\begin{minipage}[t]{0.48\linewidth}
\centering
\includegraphics[width=1\linewidth]
{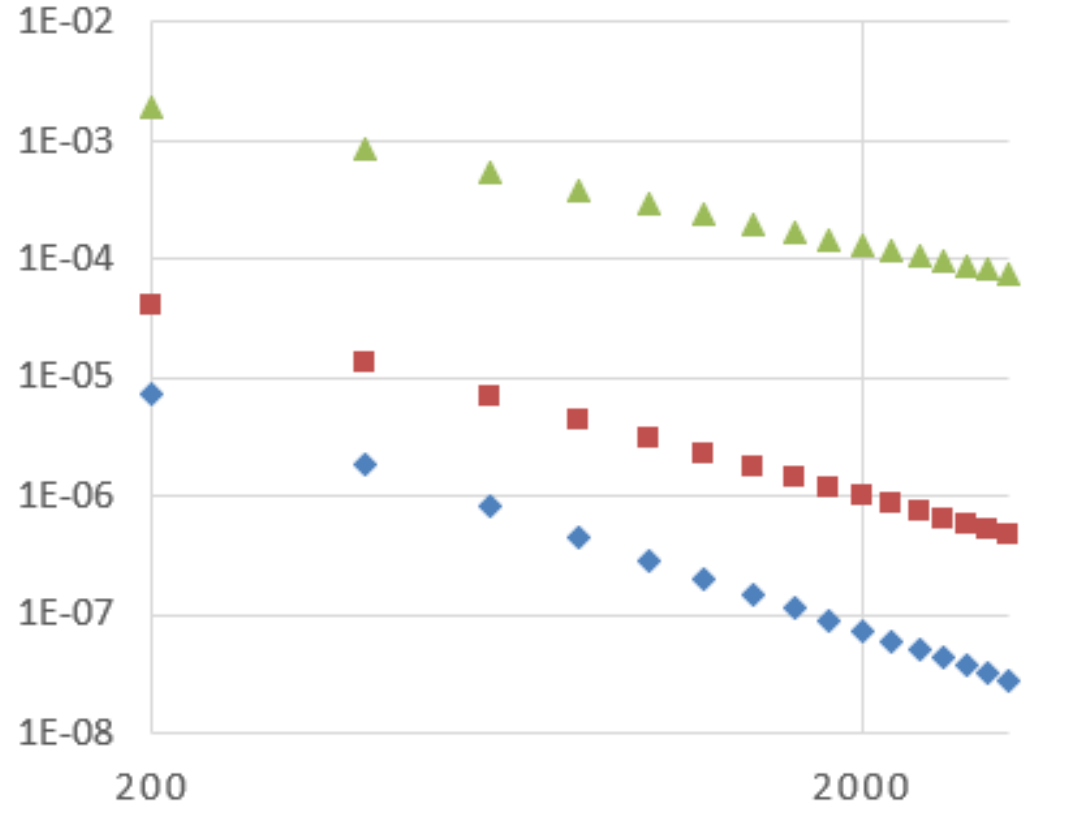}
 $b)\ \alpha=5/2$
\end{minipage}
\caption{Examples $E_{3/2}$ (left) and $E_{5/2}$ (right): $\mathcal{E},C_h,L_h^2$-norms of the error
denoted respectively by
$\triangle, \Box, \diamondsuit$, for $N=200,400,\ldots,3200$}
\label{Fig1}
\end{figure}
\begin{figure}[h]
\begin{minipage}[t]{0.48\linewidth}
\centering
\includegraphics[width=1\linewidth]
{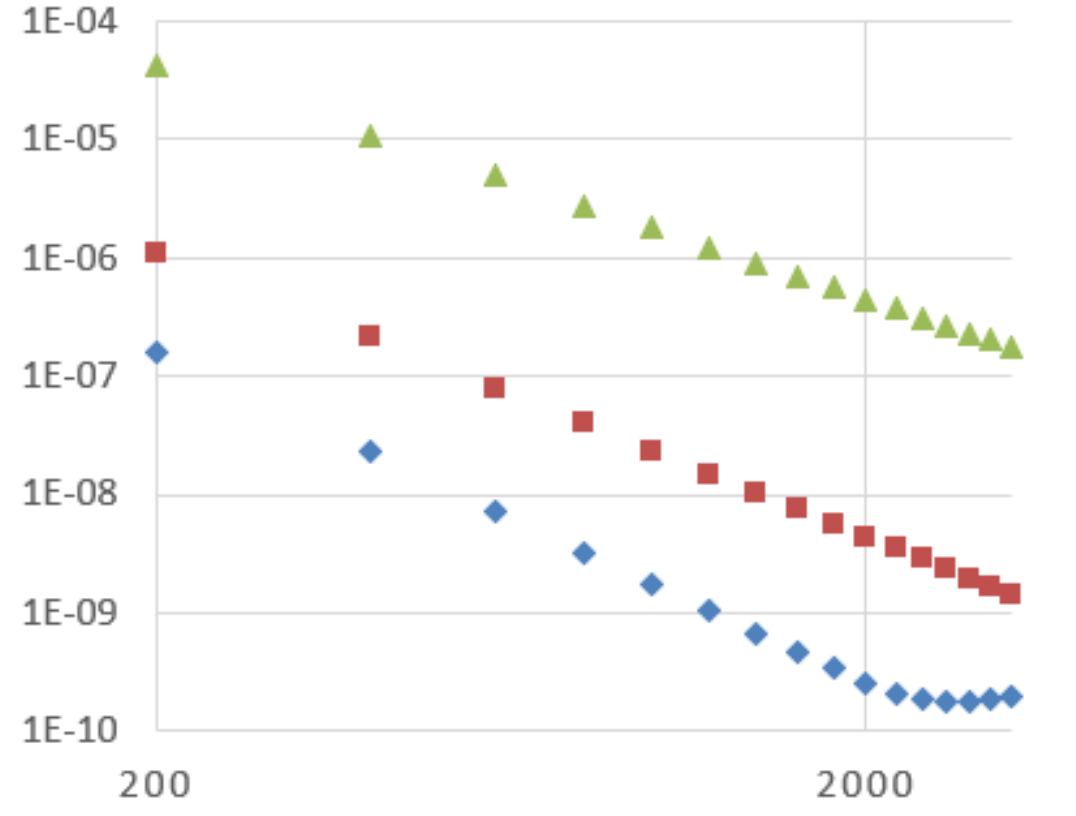}
 $a)\ \alpha=7/2$
\end{minipage}
\hfill
\begin{minipage}[t]{0.48\linewidth}
\centering
\includegraphics[width=1\linewidth]
{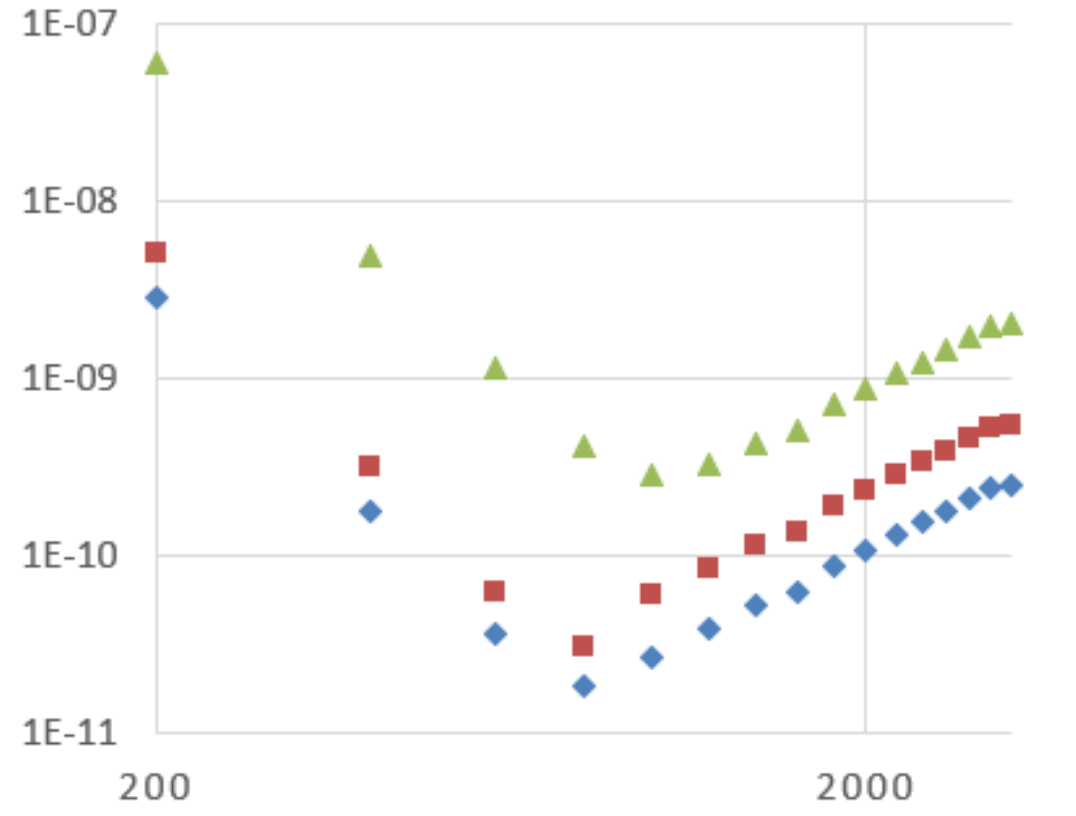}
 $b)\ \alpha=9/2$
\end{minipage}
\caption{Examples $E_{7/2}$ (left) and $E_{9/2}$ (right): $\mathcal{E},C_h,L_h^2$-norms of the error
denoted respectively by
$\triangle, \Box, \diamondsuit$, for $N=200,400,\ldots,3200$}
\label{Fig2}
\end{figure}

\par The computed $c_0$ and $\gamma_{pr}$ together with the respective theoretical orders $\gamma_{th}$ and $\gamma_{th}^{(2)}$, see \eqref{gammath}-\eqref{gammath2},
and the error norms $\|r_{N}\|$ and $\|r_{N}^{(2)}\|$ for $N=200,N_{\max}$ are collected in Table \ref{table:practical_error}.
For more visibility, here we include the error norms $\|r_{N}^{(2)}\|$ for the standard second order scheme
like \eqref{num1eq1d}-\eqref{num1ic1d} but with the multiplier $-\sigma a_1^2h_t^2$ substituted for $\tfrac{1}{12}(h_1^2-a_1^2h_t^2)$,
with the weight $\sigma=\tfrac{1}{2}$, the same $v^0$ and $f_N$ as well as $u_{1N}=q_1u_1$ for $\alpha\leq\frac32$, or $u_{1N}=u_1$ for $\alpha\geq\frac52$.
\begin{table}[t]
\caption{Numerical results for the uniform mesh}
\label{table:practical_error}
\vskip 2mm
\tabcolsep=0.15cm
\begin{tabular}{cccccccccc}
\toprule 
 $\alpha$ & $\|\cdot\|$ & $c_0$ & $\gamma_{pr}$ & $\gamma_{th}$ & $\gamma^{(2)}_{th}$ & $\|r_{200}\|$ & $\|r_{N_{\max}}\|$
 & $\|r_{200}^{(2)}\|$ & $\|r_{N_{\max}}^{(2)}\|$\\[1mm]
\toprule 
$1/2$ & $L^2_h$ & 0.514 & 0.406 &  0.4 & $1/3$ & $.595E$$-1$  & $.192E$$-1$       & $.943E$$-1$ & $.373E$$-1$\\
\toprule 
$ $ & $\mathcal{E}_h$ & 1.24  & 0.346 & 0.4 & $1/3$ & $.201E$$-0$  & $.751E$$-1$  & $.404E$$-0$ & $.172E$$-0$\\
$3/2$ & $C_h$     & 0.245 & 0.742 & 0.8 & $2/3$ & $.475E$$-2$  & $.582E$$-3$      & $.180E$$-1$ & $.294E$$-2$\\
$ $ & $L^2_h$         & 0.393 & 1.217 & 1.2 & 1     & $.635E$$-3$  & $.215E$$-4$  & $.272E$$-2$ & $.168E$$-3$\\
\toprule 
      $ $ & $\mathcal{E}_h$ & 0.924 & 1.167 & 1.2 & 1 & $.188E$$-2$  & $.745E$$-4$ & $.965E$$-2$ & $.628E$$-3$\\
$5/2$ & $C_h$ & 0.211 & 1.615 & 1.6 & $4/3$ & $.406E$$-4$  & $.461E$$-6$           & $.480E$$-3$ & $.121E$$-5$\\
      $ $ & $L^2_h$ & 0.305 & 2.007 & 2 & $5/3$ & $.734E$$-5$  & $.281E$$-7$       & $.103E$$-3$ & $.999E$$-6$\\
\toprule 
      $ $ & $\mathcal{E}_h$ & 1.49 & 1.975 & 2 & $5/3$ & $.422E$$-4$  & $.448E$$-6$& $.766E$$-3$ & $.169E$$-4$\\
$7/2$ & $C_h$ & 0.377 & 2.403 & 2.4 & 2 & $.111E$$-5$  & $.440E$$-8$               & $.993E$$-4$ & $.102E$$-5$\\
      $ $ & $L^2_h$ & 0.435 & 2.798 & 2.8 & 2 & $.160E$$-6$  & $.260E$$-9$         & $.391E$$-4$ & $.394E$$-6$\\
\toprule 
      $ $ & $\mathcal{E}_h$ & 3.23 & 2.787 & 2.8 & 2 & $.125E$$-5$ & $.261E$$-7$   & $.408E$$-3$ & $.256E$$-4$\\
$9/2$ & $C_h$           & 1.17 & 3.205 & 3.2 & 2 & $.492E$$-7$ & $.579E$$-9$       & $.144E$$-3$ & $.906E$$-5$\\
      $ $ & $L^2_h$         & 1.21 & 3.601 & 3.6 & 2 & $.628E$$-8$ & $.427E$$-10$  & $.858E$$-4$ & $.536E$$-5$\\
\toprule 
     $ $ & $\mathcal{E}_h$ & 11.2 & 3.597 & 3.6 & 2 & $.593E$$-7$ & $.114E$$-8$    & $.125E$$-2$ & $.139E$$-3$\\
$11/2$   & $C_h$ & 8.02 & 3.997 & 4   & 2 & $.508E$$-8$ & $.631E$$-10$             & $.502E$$-3$ & $.558E$$-4$\\
       $ $ & $L^2_h$       & 3.77 & 3.966 & 4   & 2 & $.285E$$-8$ & $.370E$$-10$   & $.310E$$-3$ & $.347E$$-4$\\
\toprule 
\end{tabular}
\end{table}

\par The main observation is the nice agreement between $\gamma_{pr}$ and $\gamma_{th}$ for all three norms in all Examples $E_{\alpha}$, thus the sensitive dependence of $\gamma_{pr}$ on the data smoothness order $\alpha$ becomes quite clear.
This agreement is mainly better for the first and second norms \eqref{3norms} (similarly to \cite{ZKMMA2018}).
Notice that $\gamma_{th}^{(2)}/\gamma_{pr}$ and the error $\|r_{200}\|$ in each norm decrease rapidly as $\alpha$ grows.
Clearly the errors  $\|r_{N}\|$ are much smaller than $\|r_{N}^{(2)}\|$ for $N=200$ and $N_{\max}$ especially as $\alpha$ grows.
This demonstrates the essential advantages of the 4th approximation order scheme over the 2nd order one in the important case of non-smooth data as well.
This is essential, in particular, in some optimal control problems~ \cite{TVZ18}.

\par We also remind the explicit scheme \eqref{explicit1}-\eqref{explicit2}.
For the same $X$ and $a$ but $h_t=h/a$ and $T=Mh_t>1$, for example, the $C_h$-norm of the error equals $0.311E$$-14$ even for $N=20$ and $M=10$ already in Example $E_{3/2}$; thus clearly it is caused purely by the round-off errors.

\smallskip\par\textbf{\ref{numerexperiments}.2.} Also we analyze numerically
scheme \eqref{num3eq nonuni in x} and \eqref{num3ic nonuni} (with $f_N^0=s_Nf^0+\frac23(f^{\tau/2}-f^0)$)
on non-uniform spatial meshes such that
$x_k=\varphi(\frac{k}{N})-\frac{X}{2}$, $0\leq k\leq N$, and $h_k=x_k-x_{k-1}$.
Here $\varphi\in C[0,1]$ is a given increasing node distribution function with the range $\varphi([0,1])=[0,X]$.
We take again $X=T=1$ and $a=\frac{1}{\sqrt{5}}$ but consider only the smooth (analytic) exact solution $u$ for the data
\begin{gather*}
 u_0(x)=\sin(2\pi(x+0.5)),\ u_1(x)=4\sin(3\pi(x+0.5)),\ f(x,t)=e^{x+0.5-t},
\\[1mm]
 g_0(t)=\tfrac{1}{2a}\big(\tfrac{1}{a+1}e^{at}+\tfrac{1}{a-1}e^{-at}-\tfrac{2a}{a^2-1}e^{-t}\big)\ (a\neq 1),\ \
 g_1(t)=eg_0(t).
\end{gather*}

\par We base on the practical stability condition $h_t^2\frac{a^2}{h_{\min}^2}\leq\half$ with $h_{\min}=\min_{1\leq k\leq N}h_k$ (cp. \eqref{condhth2} for $C_0=1$ and $\ve_0^2=\half$), thus we set $M=M_0:=\big\lfloor\frac{\sqrt{2}aT}{h_{\min}}\big\rfloor$, where $\lfloor b\rfloor$ is the maximal integer less or equal $b$.
It turns out to be accurate in practice, see below.
We take $N=50,100,\ldots,1000$.

In Table \ref{table:practical_error num}, the error behavior in the $C_h$ norm is represented for several functions $\varphi_l$, $0\leq l\leq 6$.
Clearly $\varphi_0(t)=t$ sets the uniform mesh and is included for comparison only.
Notice that $\varphi_3'(0)=0$ whereas $\varphi_l'(+0)=+\infty$, $l=4,5,6$; both cases are more complicated than the standard one $0<\underline{c}\leq\varphi_l'(\xi)\leq \bar{c}$ on $[0,1]$ , $l=1,2$, in the existing theory \cite{Z15}.

\par The error orders $\gamma_{pr}$ are close to 4 for $0\leq l\leq 3$ but decrease down to $2.411$ as in $\varphi_l(\xi)=\xi^{a_l}$ the power $a_l=\frac34,\frac58,\frac12$ diminishes, $l=4,5,6$.
Thus the approximation orders 3 or 4, see Section \ref{nonunif_mesh}, are not always the practical error orders as well.
For $l=2,6$, the values of $c_0$ and $\gamma_{pr}$ are marked by $^*$ meaning that the results are yet too rough for $N=50,100,150$ and thus ignored in their computation.
For any $l$, the graphs of $\log_{10}\|r_N\|_{C_h}$ versus $\log_{10}N$ are very close to straight lines (omitted for brevity).

\par The mesh data $\frac{h_{\max}}{h_{\min}},
\rho_{\min}:=\min_{1\leq k\leq N-1}\frac{h_{k+1}}{h_k},
\rho_{\max}:=\max_{1\leq k\leq N-1}\frac{h_{k+1}}{h_k}$ and $\frac{M}{N}$, all for $N=800$ only, are also included into the table.
Note that condition \eqref{step ratio cond} is violated for $l=3,5,6$, but this does not essentially affect the results.
For $l=1$, $\rho_{\min}=\rho_{\max}$ since the steps $h_k$ form a geometric progression.
Also  $\varphi_l$ is strictly convex (or concave) on $[0,1]$ for $l=1,3$ (or $l=2,4,5,6$), accordingly $h_k=\varphi_l'(\xi_k)$, where
$\xi_k\in(\frac{k-1}{N},\frac{k}{N})$,
increases and $\rho_{\min}>1$ (or decreases and $\rho_{\max}<1$) as $k$ grows.
The ratios $\frac{M}{N}$ are not high except $l=1,3$.

\par Taking smaller $M$ by replacing $\sqrt2$ with $\frac{1}{\sqrt2}$ in the above formula,
for $l=0,1$ (the cases of the uniform and non-uniform meshes), leads us to highly unstable computations for $N\geq 100$: the $C_h$-norm of numerical solutions grows exponentially.
\begin{table}[t]
\caption{Numerical results for non-uniform spatial meshes, with $\varphi_2(\xi)=\frac{\ln (60\xi+1)}{\ln 61}$}
\label{table:practical_error num}
\vskip 2mm
\tabcolsep=0.14cm
\begin{tabular}{cccccccccc}
\toprule 
 $l$
 & $\varphi_l(\xi)$ & $c_0$ & $\gamma_{pr}$
 & $\|r_{200}\|_{C_h}$
 & $\|r_{400}\|_{C_h}$
 & $\|r_{800}\|_{C_h}$
 & $\frac{h_{\max}}{h_{\min}}$
 & $[\rho_{\min},\rho_{\max}]$
  & $\frac{M}{N}$
 \\[1mm]
\toprule 
0 & $\xi$ & $42.0$ & $4.001$
& $.262E$$-7$ & $.164E$$-8$ & $.103E$$-9$  & $1$ & $[1,1]$ & $.631$\\
\toprule 
1 & $\frac{e^{5\xi}-1}{e^5-1}$ &$15878$ & $3.988$
&  $.107E$$-4$ & $.668E$$-6$ & $.418E$$-7$  & $147.5$ & $[1.006,1.006]$ & $18.6$\\[1mm]
\toprule 
2 & $\varphi_2(\xi)$ & $966736^*$ & $3.945^*$
&  $.784E$$-3$ & $.533E$$-4$ & $.341E$$-5$  & $58.78$ & $[.9325,.9988]$ & $2.64$\\
\toprule 
3 & $\xi^{3/2}$  &$161.7$ & $4.001$
&  $.101E$$-6$ & $.629E$$-8$ & $.392E$$-9$  & $42.41$ &
$[1.001,1.828]$
& $17.9$\\
\toprule 
4 & $\xi^{3/4}$  &$495.2$ & $3.600$
& $.265E$$-5$ & $.216E$$-6$ & $.173E$$-7$  & $7.090$ & $[.6818,.9997]$ & $.843$\\
\toprule 
5 & $\xi^{5/8}$  &$468.7$ & $3.019$
& $.556E$$-4$ & $.672E$$-5$ & $.792E$$-6$  & $19.62$ & $[.5422,.9995]$ & $1.01$\\ 
\toprule 
6 & $\xi^{1/2}$ &$427.0^*$ & $2.411^*$
&  $.118E$$-2$ & $.230E$$-3$ & $.427E$$-4$  & $56.55$ & $[.4142,.9994]$ & $1.26$\\
\toprule 
\end{tabular}
\end{table}

\smallskip
\par \textbf{Acknowledgements}
\par The publication was prepared within the framework of the Academic Fund Program at the
National Research University Higher School of Economics (HSE) in 2019--2020 (grant no. 19-01-021)
and by the Russian Academic Excellence Project ``5-100''
as well as by the Russian Foundation for the Basic Research, grant no.~19-01-00262.

\newpage

\bibliography{ZlotnikKireeva2020_corr_2}
\end{document}